\documentclass[a4paper,12pt,reqno]{amsart}
\usepackage{geometry}
\usepackage[utf8]{inputenc}

\usepackage{amsfonts}  
\usepackage{amsthm}    
\usepackage{amsmath}   
\usepackage{bm}        
\usepackage{amssymb}   
\usepackage{commath}	

\usepackage[dvipsnames, table]{xcolor}	
\usepackage{graphicx}  
\usepackage{tikz}      
\usepackage{tikz-cd}   
\usepackage{pgfplots}  
\pgfplotsset{compat=1.18}
\usetikzlibrary{arrows.meta, positioning, calc, math}  
\tikzset{>={Classical TikZ Rightarrow[width=7pt, length=5pt]}}
\usetikzlibrary{cd}  

\usepackage{array}     
\usepackage{xtab}      
\usepackage{url}

\usepackage{comment}
\usepackage[T1]{fontenc} 
\usepackage{lmodern} 
\usepackage{multicol}
\usepackage{pdflscape, adjustbox}	
\usepackage{parskip}       
\usepackage{mathrsfs}      
\usepackage{fancyhdr}      
\usepackage[hidelinks]{hyperref}  
\usepackage{mathtools}     
\usepackage[shortlabels]{enumitem}  
\usepackage{witharrows}    
\usepackage{float}         
\usepackage{makecell}
\usepackage{cleveref}	

\newcommand{\Z}{\ensuremath{\mathbb{Z}}}
\newcommand{\Q}{\ensuremath{\mathbb{Q}}}
\newcommand{\N}{\ensuremath{\mathbb{N}}}
\newcommand{\R}{\ensuremath{\mathbb{R}}}
\newcommand{\F}{\ensuremath{\mathbb{F}}}
\newcommand{\p}{\ensuremath{\mathfrak{p}}}
\newcommand{\A}{\ensuremath{\mathfrak{a}}}
\newcommand{\g}{\ensuremath{\mathfrak{g}}}
\newcommand{\OO}{\ensuremath{\mathcal{O}}}

\DeclareMathOperator{\Ann}{ann}

\DeclareMathOperator{\Min}{min}

\DeclareMathOperator{\Mod}{mod}
\DeclareMathOperator{\End}{End}

\DeclareMathOperator{\Lcm}{lcm}

\newtheorem{prop}{Proposition}[section]

\newtheorem{theorem}[prop]{Theorem}
\newtheorem{lemma}[prop]{Lemma}

\newtheorem{cor}[prop]{Corollary}

\theoremstyle{definition}
\newtheorem{remark}[prop]{Remark}
\newtheorem{example}[prop]{Example}

\newtheorem{defn}[prop]{Definition}

\numberwithin{equation}{section}

\crefname{prop}{Proposition}{Propositions}
\Crefname{prop}{Proposition}{Propositions}

\crefname{cor}{Corollary}{Corollaries}
\Crefname{cor}{Corollary}{Corollaries}

\crefname{defn}{Definition}{Definitions}
\Crefname{defn}{Definition}{Definitions}

\makeatletter
\renewcommand{\textcolor}[2]{#2}
\makeatother
\include{contents}
\begin{document}
\author{Edison H L Au-Yeung}
\title[Explicit valuation of elliptic nets for elliptic curves with CM]{Explicit valuation of elliptic nets for elliptic curves with complex multiplication}
\address{Mathematics Institute, Zeeman Building, University of Warwick, Coventry CV4 7AL, England}
\email{hang.au-yeung@warwick.ac.uk}
\date{\today}
\thanks{The author is supported by the Warwick Mathematics Institute Centre for Doctoral Training, and gratefully acknowledges funding from the University of Warwick.}
\begin{abstract}
Division polynomials associated to an elliptic curve $E/K$ are polynomials $\phi_n, \psi_n^2$ that arise from the sequence of points $\{nP\}_{n \in \N}$ on this curve. If one wishes to study $\Z$--linear combination of points on $E(K)$, we can use net polynomials $\Phi_{\bm{v}}, \Psi_{\bm{v}}^2$ which are higher--dimensional \textcolor{red}{analogues} of division polynomials. It turns out they are also elliptic nets, an $n$--dimensional array with values in $K$ satisfying the same nonlinear recurrence relation that division polynomials do as well. Now further assume the elliptic curve $E/K$ has complex multiplication by an order of a quadratic imaginary field $F \subseteq K$, we will prove a formula for the common valuation of $\Phi_{\bm{v}}$ and $\Psi_{\bm{v}}^2$ associated to multiples of points by elements of an order in $F$. As an application, we will use the formula to show that elliptic divisibility sequences associated to multiples of points indexed by elements of an order also satisfy a recurrence relation when indexed by elements of an order, subject to certain conditions on the indices.

\end{abstract}
\maketitle
\setcounter{tocdepth}{1}
\tableofcontents
\section{Introduction}
\subsection{Motivation}
Let $E/\Q$ be an elliptic curve defined over the rationals with the following Weierstrass equation:
\begin{equation}\label{EC}
E\colon y^2 + a_1xy + a_3y = x^3 + a_2x^2 + a_4x + a_6;\, a_i \in \Z
\end{equation}

Let $P=(x, y)$ be a non-torsion point on $E(\Q)$, $n$ an integer. Then we can express the point $nP$ with the $n$-th division polynomials or the numerical values directly:
\begin{equation}\label{expression}
x(nP)=\frac{\phi_{n}(x(P))}{\psi_{n}^{2}(x(P))}=\frac{A_{n}(E, P)}{B_{n}^{2}(E, P)}\,\, \text{with} \,\, (A_{n}, B_{n})=1 \,\, \text{and} \,\, B_{n}>0.
\end{equation}
with $A_{n}(E, P)$ and $B_{n}(E, P)$ two coprime integers and $B_{n}(E, P) \geq 0$. In \cite[Exercise 3.7]{Silverman1}, Silverman mentioned that the division polynomials $\{\psi_{n}\}_{n \in \N}$ are a sequence of rational functions on $E$ that satisfy the recurrence relation 
\begin{equation}\label{recurrence}
\psi_{n+m}\psi_{n-m}\psi_r^2 = \psi_{m+r}\psi_{m-r}\psi_n^2 -\psi_{n+r}\psi_{n-r}\psi_m^2 \,\, \text{for any}\,\, n>m>r. 
\end{equation}
Division polynomials were introduced to compute scalar multiplications of points on elliptic curves. However, when evaluated at a point on an elliptic curve, these polynomials are not usually coprime. This has led to an interest in the quantity \(\min(\nu(\phi_{n}(P)), \nu(\psi_n^2(P)))\) for a fixed discrete valuation $\nu$, commonly referred to as the \emph{cancellation exponent}. By understanding or estimating this quantity, one can address problems related to Siegel's famous theorem, which states that any elliptic curve \(E/\Q\) has only finitely many integral points.

For instance, Ayad used this concept to compute \(S\)-integral points of an elliptic curve of rank 1 (see \cite[Sections 6--8]{Ayad}). Subsequently, Ingram \cite{Ingram} estimated the height of such points by bounding \(|\psi_{n}(P)|\) in terms of the denominator of \([n]P\), a process that involves analysing the valuations of \(\psi_{n}(P)\) at each prime. Building on this work, Stange \cite{Stange2} refined Ingram's results, obtaining a bound that depends on the height ratio \(h(E)/\hat{h}(P)\).
Meanwhile, Amir \cite{Amir} focused on Mordell curves \(E_B: y^2 = x^3 + B\), demonstrating that, with finitely many exceptions, there are at most two values \(1 < n < 11\) for which \([n]P\) is integral. While this significantly improves the bounds, it still depends on the valuation of division polynomials.

There have already been several attempts to find a formula for the cancellation exponent between the division polynomials \(\psi_n^2\) and \(\phi_n\). For example, Cheon \cite{Cheon} provides a recursive formula at singular primes, utilising the set of points with nonsingular reduction \(E_{0}(K) \subseteq E(K)\), which is a subgroup of finite index. Stange \cite{Stange2}, on the other hand, presents a closed formula whose parameters depend on the reduction properties of \(P\) and \(E\). Yabuta and Voutier gave a formula in terms of the Kodaira symbol in \cite{Yabuta}, while Naskr\k{e}cki and Verzobio extended their results to elliptic curves over function fields in \cite{Verzobio2}. Understanding the cancellation exponent is particularly useful for the study of elliptic divisibility sequences (EDSs). This is closely related to the application to Siegel's Theorem mentioned above, as finding all \(n\) for which \([n]P\) is integral is equivalent to finding all \(n\) such that the denominator of the coordinates of the \(n\)-th term, \(B_n\), equals \(1\).

If the elliptic curve \(E\) now has complex multiplication (CM), it is natural to ask for objects with properties like division polynomials and EDSs, indexed by elements of an order of a quadratic imaginary field \(\End(E)\), generated by the basis set \(\{1, \omega\}\), instead of just by integers. There have already been some attempts to extend definitions and results related to EDSs and division polynomials to the case where \(E\) has complex multiplication. For instance, Streng \cite{Streng} extended the definition of EDSs for elliptic curve points and established results concerning primitive divisors. Meanwhile, Satoh \cite{Satoh} proved the existence of generalised division polynomials for elliptic curves with complex multiplication and demonstrated that these polynomials satisfy certain recurrence relations. Therefore, if one also derives a formula for the cancellation exponent of CM division polynomials, it becomes possible to extend all the aforementioned applications to CM elliptic curves, taking the additional CM points into account. 

In this paper, we define CM division polynomials using elliptic nets, as defined in \Cref{elliptic net}, rather than Satoh's approach for two reasons: first, the behaviour of the generalised division polynomials defined by Satoh in \cite{Satoh} depends on whether or not the points in the kernel \(E[\alpha]\) sum to zero and whether $2$ ramifies, splits or remains inert (interested readers are encouraged to see \cite{Satoh} or \cite{Stange4}). Second, elliptic nets are a `higher-rank' generalisation of EDSs. In \cite{Stange2}, it was noted that analogous results for bounding integral points among linear combinations \([n]P + [m]Q\) currently seem out of reach. However, if an elliptic curve \(E\) has complex multiplication by \(\Z[\omega]\), one can consider the \(\Z\)-linear combination \textcolor{red}{of points} \([n]P + [m]\omega P\), which is effectively a multiple of a point \([n + m\omega]P\). Specifically, the author is now attempting to apply the explicit common valuation formula to use in the same spirit by generalising an argument of Ingram in \cite{Ingram}, with the aim of producing a bound on $n+m\omega$ with the second largest norm such that the point $[n+\omega]P$ is integral. Thus, we believe that defining CM elliptic divisibility sequences using elliptic nets could lead to improved bounds on counting problems and serve as a prototype for a linear combination version of the integral point problem, which is the motivation behind our interest in this subject. 

\subsection{Main results}
In this paper, we extend Cheon and Hahn's result for the cancellation exponent below to curves with complex multiplication, allowing us to index the `division polynomials' by an order in a quadratic imaginary field \(\End(E)=\Z[\omega]\), instead of just integers. For the convenience of reader, their result is repeated here. Let $K$ be a number field and $\nu$ the discrete valuation related to a prime ideal $\p$ of $\OO_K$ with $\nu(\pi)=1$ for a uniformiser $\pi$ of $\p$. Define 
\[
g_{n, \nu}(P)=\min\left(2\nu\left(\psi_{n}(P)\right), \nu\left(\phi_{n}(P)\right)\right).
\]
\begin{theorem}\label{asymptotic}\cite[Theorem 4]{Cheon}
Let E be an elliptic curve defined by a Weierstrass equation with coefficients in $\OO_{K}$. Let $P$ be a non-torsion point of $E(K)$ and assume $\nu(x(P)) \geq 0$. Let $r>1$ be the order of the point $P$ in the group $E(K)/E_{0}(K)$. For $n>0$, $g_{n, \nu}(P)$ is asymptotically equal to a quadratic function. More precisely,
\begin{equation}
g_{n, \nu}(P) = \begin{cases}
\mu t^{2}, & \text{if $n=tr$.} \\
4\mu t^{2} \pm 2\left(2\nu\left(\frac{\psi_{k}(P)}{\psi_{r-k}(P)}\right)+\mu\right)t+ 2 \nu(\psi_{k}(P)\textcolor{red}{)}, &\text{if $n=2tr\pm k$ with $1 \leq k < r$},
\end{cases}
\end{equation}
where $\mu = g_{r, \nu}(P)$. 
\end{theorem}

\begin{remark}\label{error}
Although the final result (Theorem \ref{asymptotic}) is correct, the proof by Cheon and Hahn has a gap. The source of inaccuracy originates from Lemma 1 of the paper, which claims that if a non-torsion point $P \in E_{0}(K)$ satisfies $\nu(x)<0$, then $\nu\left(x(mP)\right)=\nu\textcolor{red}{\left(x(P)-2\nu(m)\right)}$. This result is only true under additional hypotheses, as stated in \cite[Remark 6.5]{Stange2}. This slip affects Theorem 3, and hence the proof of Theorem 3 in \cite{Cheon}. However, as suggested in \cite[Remark 3.6]{Verzobio}, this proof can be fixed easily using the subgroup filtration in (\ref{filtration}) below, which is why Theorem \ref{asymptotic} is still true.  
\end{remark}

The notations for our main result are as follows: let $K$ be a number field, $E/K$ be an elliptic curve defined by a Weierstrass equation with coefficients in $\OO_{K}$ and complex multiplication by an order $\End(E)=\Z[\omega]$ in a quadratic imaginary field $F \subseteq K$, so $[z]P \in E(K)$ for $z \in \End(E)$. Let $P=(x, y)$ be a non-torsion point of $E(K)$ and $z$ be an element in the endomorphism ring $\End(E)$ of the curve $E$, so we can write
\begin{equation}\label{expression}
x([z]P)=\frac{\phi_{z}(P)}{\psi_{z}^{2}(P)} \text{ and } x(\textcolor{red}{[z]}P)\OO_{K}=A_{z}(P)D_{z}(P)^{-2},
\end{equation}
where $\psi_{z}$ and $\phi_{z}$ denote the corresponding net polynomials (see \ref{nett}) and $A_{z}(P), D_{z}(P)$ are coprime $\OO_{K}$--ideals. If $K$ has class number $1$, then we can also write 

\[
x([z]P)=\frac{\phi_{z}(P)}{\psi_{z}^{2}(P)}=\frac{A_{z}(E, P)}{B_{z}(E, P)^2} 
\]
with $A_z$, $B_z$ coprime elements in $K$. In this case, we have $D_{z}(P)=(B_z)$, where $B_{z}$ is a generator to be picked. To give our main result, we make two definitions regarding valuations. 

\begin{defn}\label{Valuations}
Let $\p$ be a prime ideal in $\OO_{K}$, lying above a rational prime $p$. For $\alpha \in K^{\times}$, write
\[
\alpha\OO_{K}=\p^{a_\p}\prod_{i=1}^{n}\mathfrak{q}_{i}^{a_i}
\]
as the unique factorisation of $\alpha \OO_K$, where $\mathfrak{q}_i$ are prime ideals for $1\leq i\leq n$. We define the valuation $\nu$ associated to $\p$ to be the function
\begin{align*}
\nu \colon K^{\times} &\rightarrow \Z \\
\alpha &\mapsto a_{\p}
\end{align*}
so $\nu$ is a normalised discrete valuation. For an ideal $\mathfrak{a}$ in $\OO_K$ with unique factorisation
\[
\mathfrak{a}=\p^{a_\p}\prod_{i=1}^{n}\mathfrak{q}_{i}^{a_i},
\]
we further define 
\[
d_{\mathfrak{p}}(\mathfrak{a}) \coloneqq a_\p.
\]
\end{defn}

\begin{defn}\label{gzP}
Let $E/K$ be an elliptic curve defined by a Weierstrass equation with coefficients in $\OO_{K}$ and which has complex multiplication by an order $\End(E)=\Z[\omega]$ in a quadratic imaginary field $F\subset K$. Let $\nu$ be a normalised discrete valuation with respect to a prime ideal $\p$ of $\OO_{K}$ (so $\nu(\p)=1$) and $z \in \Z[\omega]$. For $P \in E(K)$, if $[z]P\neq O$, then we define 
\[
g_{z, \nu}(P)\coloneqq\Min(2\nu\left(\psi_{z}(P)\right), \nu\left(\phi_{z}(P)\right))
\]
to be the common valuation of the net polynomials. Moreover, if $[z]P=O$, we set $g_{z, \nu}(P)=0$. 
\end{defn}

\begin{theorem}\label{CM gcd}
With the notations above, $g_{z, \nu}(P)$ is asymptotically equal to a quadratic function. More precisely,
\begin{enumerate}
\item If $P$ is non--singular modulo $\nu$, then
\[
\begin{aligned}
g_{z, \nu}(P) 
&=-2\nu\left(F_{\bm{z}}(\bm{P})\right) \\
&= ( z_1 z_2-z_1^2)\max(0, -\nu(x(P))) 
   + (z_1 z_2-z_2^2)\max(0, -\nu(x([\omega]P))) \\
&\quad - (z_1 z_2)\max(0, -\nu(x([1+\omega]P))),
\end{aligned}
\]
where $F_{\bm{z}}(\bm{P})$ is as defined in (\ref{F}).

\item If $P$ is singular modulo $\nu$ (so $\nu(x(P)) \geq 0$), 
let $r$ be an element of the annihilator of $P$ in the 
$\Z[\omega]$--module $E(K)/E_{0}(K)$, 
$\alpha = a + b\omega \in \Z[\omega]$. Then
\begin{equation}
g_{z, \nu}(P) = 
\begin{cases}
(a^{2} - ab)\mu + (b^{2} - ab)\mu_{\omega} + (ab)\mu_{1+\omega}, 
& \text{if $z = \alpha r \in \Ann(P)$,} \\[0.5em]
2 \nu\!\left(\psi_{\beta}\right) 
\pm 2a\!\textcolor{red}{\left(\nu\!\left(\dfrac{\psi_{\beta}}{\psi_{r - \beta}}\right) 
   + \dfrac{\mu}{2}\right)} 
\pm 2b\!\textcolor{red}{\left(\nu\!\left(\dfrac{\psi_{\beta}}{\psi_{r \omega - \beta}}\right) 
   + \dfrac{\mu_{\omega}}{2}\right)} \\[0.5em]
\quad + (a^{2} - ab)\mu + (b^{2} - ab)\mu_{\omega} + (ab)\mu_{1+\omega}, 
& \text{if $z = \alpha r \pm \beta \not\in \Ann(P)$.}
\end{cases}
\end{equation}
where 
\[
\mu(r) = g_{r, \nu}(P), \quad 
\mu_{\omega}(r) = g_{r\omega, \nu}(P), \quad 
\mu_{1+\omega}(r) = g_{r(1+\omega), \nu}(P).
\]
\end{enumerate}

If $\End(E)=\Z[\omega]$ is a principal ideal domain and $r$ generates $\Ann(P)$, then one of the two cases above always holds. 
\end{theorem}
\begin{remark}
In the case of non-singular reduction, this is a result proved by \cite[Proposition 1.7]{Akbary}.
\end{remark}
\begin{remark}\label{r dependency}
We would like to emphasise that in the case of singular reduction, the valuation formula in general does depend on which element $r$ of the annihilator ideal is chosen. As we vary $r$, the values of $\mu(r), \mu_{\omega}(r)$ and $\mu_{1+\omega}(r)$ vary as well. 
\end{remark}

In \Cref{Examples}, we provide two examples to illustrate this theorem. For \Cref{Example 1}, we consider the curve
\(
E \colon y^{2} = x^{3} - 2x
\)
over the field \( \mathbb{Q}(i) \), which has complex multiplication by \( \mathbb{Z}[i] \).
We choose \(\bm{P} = (P, iP)\), where \( P = (-1, 1) \) and \( \textcolor{red}{[i]} P = (1, i) \).
Both points have good reduction everywhere, so the common valuations of the net polynomials are correctly predicted by \Cref{CM gcd}(1). On the other hand, in \Cref{Example 2} we consider the elliptic curve
\(
E \colon y^{2} = x^{3} + x^{2} - 3x + 1
\)
over the field \( \mathbb{Q}(\sqrt{-2}) \), which has complex multiplication by
\( \mathbb{Z}[\sqrt{-2}] = \mathbb{Z}[\omega] \).
We choose \(\bm{P} = (P, [\sqrt{-2}]P)\), where \( P = (-1, 2) \) and
\(
[\sqrt{-2}]P = \left( -\frac{1}{\omega^{2}}, \frac{1}{\omega^{3}} \right).
\)
The points \( P \) and \( [1+\sqrt{-2}]P \) have singular reduction modulo \( \sqrt{-2} \), while \( [\sqrt{-2}]P \) does not.
Therefore, we must use \Cref{CM gcd}(2) to correctly determine the common valuations of the net polynomials at the prime \( \sqrt{-2} \). \textcolor{red}{If one tries to use \Cref{CM gcd}(1) instead, the resulting value will contradict the examples listed in \Cref{Table 2}.}

One reason we choose to generalise Cheon and Hahn’s recursive formula in \cite[Theorem 4]{Cheon} rather than the closed formula by Stange in \cite[Theorem 6.1, 7.1]{Stange2} is that, in addition to the motivations outlined earlier, it has an interesting application. In \cite{Verzobio}, Verzobio used Cheon’s exponential cancellation formula to demonstrate that all elliptic divisibility sequences satisfy a recurrence relation in the rational case, subject to conditions on the indices. In this paper, we extend this result to further highlight the utility of \Cref{CM gcd}.

\begin{theorem}\label{CM recurrence}
Let $E$ be an elliptic curve defined by a Weierstrass equation with integer coefficients in $K$, a number field with class number $1$, and which has complex multiplication by $\End(E)=Z[\omega]$.  Let $P \in E(K)$ be a non-torsion point. Consider the sequence of denominator ideals $\{D_{[\alpha]P}\}_{\alpha \in \Z[\omega]}$. Let $\alpha, \beta$ and $\gamma$ be elements of $\Z[\omega]$\textcolor{red}{. If} two of them are \textcolor{red}{elements of $\mathfrak{M}(P)$, an ideal} to be defined in \Cref{M(P)}, then we can choose a sequence of generators $D_{[\alpha]P}=(B_{\alpha})$ such that
\[
B_{\alpha+\beta}B_{\alpha-\beta}B_{\gamma}^2=B_{\alpha+\gamma}B_{\alpha-\gamma}B_{\beta}^2-B_{\beta+\gamma}B_{\beta-\gamma}B_{\alpha}^2.
\] 
\end{theorem}

This follows from the more general result below.

\begin{theorem}
\label{general CM recurrence}
Let $E$ be an elliptic curve defined by a Weierstrass equation with integer coefficients in $K$ and which has complex multiplication by $\End(E)=\Z[\omega]$. Let $P \in E(K)$ be a non-torsion point. Consider the sequence of denominator ideals $\{D_{\alpha}\}_{\alpha \in \Z[\omega]}$, and the sequence $\{\psi_{\alpha}\}=\{\psi_{\alpha}(P)\}_{\alpha \in \Z[\omega]}$ whose terms are elements of $K$. Let $\alpha, \beta$ and $\gamma$ be elements of $\Z[\omega]$.

Suppose that at least two of $\alpha, \beta$ and $\gamma$ belong to the ideal $\mathfrak{M}(P)$ defined in \Cref{M(P)}; and suppose further that either
\begin{enumerate}[(a)]
\item $\Z[\omega]$ is a principal ideal domain, or
\item there exists an element $M \in \mathfrak{M}(P)$ such that \textcolor{red}{for} each of $\alpha, \beta$ and $\gamma$, \textcolor{red}{it} is either a multiple of $M$, or coprime to $M$.
\end{enumerate}
Then there is a sequence $\{\mathfrak{g}_{\alpha}\}_{\alpha \in \Z[\omega]}$ whose terms are \textcolor{red}{fractional} ideals of $\OO_K$ with 
\begin{align*}
D_{\alpha}&=\psi_{\alpha}\mathfrak{g}_{\alpha}^{-1}, \\
\psi_{\alpha+\beta}\psi_{\alpha-\beta}\psi_{\gamma}^2 &= \psi_{\alpha+\gamma}\psi_{\alpha-\gamma}\psi_{\beta}^2 - \psi_{\beta+\gamma}\psi_{\beta-\gamma}\psi_{\alpha}^2, \\
\mathfrak{g}_{\alpha+\beta}\mathfrak{g}_{\alpha-\beta}\mathfrak{g}_{\gamma}^2 &= \mathfrak{g}_{\alpha+\gamma}\mathfrak{g}_{\alpha-\gamma}\mathfrak{g}_{\beta}^2 = \mathfrak{g}_{\beta+\gamma}\mathfrak{g}_{\beta-\gamma}\mathfrak{g}_{\alpha}^2.
\end{align*}
\end{theorem}

In the case of \textcolor{red}{$\OO_K$} being a principal ideal domain, the sequence of generators $\{B_{\alpha}\}_{\alpha \in \Z[\omega]}$ is not in general an EDS in the traditional sense, since it need not satisfy the recurrence relation (\ref{recurrence}). \textcolor{red}{More precisely, \Cref{general CM recurrence} shows that $\{B_{\alpha}\}_{\alpha \in \Z[\omega]}$ does not satisfy the recurrence relation (\ref{recurrence})} if $P$ admits primes of bad reduction. In \Cref{Examples}, \Cref{Example 1} considers the curve \(
E \colon y^{2} = x^{3} - 2x
\)
over the field \( \mathbb{Q}(i) \) with complex multiplication by $\Z[i]$ and the point $P=(1, -1)$. The point $P$ has good reduction at all primes, so \Cref{CM recurrence} tells us that \textcolor{red}{$\mathfrak{M}(P)=\OO_K$} or equivalently, the recurrence relation of the sequence of generators $\{B_{\alpha}\}_{\alpha \in \Z[i]}$ holds for all $\alpha \in \Z[i]$. On the other hand, for the curve \(
E \colon y^{2} = x^{3} + x^{2} - 3x + 1
\)
over the field \( \mathbb{Q}(\sqrt{-2}) \), which has complex multiplication by
\( \mathbb{Z}[\sqrt{-2}] = \mathbb{Z}[\omega] \), and the point \( P = (-1, 2) \) in \Cref{Example 2}, $P$ has bad reduction at the prime $\sqrt{-2}$. According to \Cref{general CM recurrence}, the ideal $\mathfrak{M}(P)$ is $\left(\sqrt{-2}\right)$ and we require two of $\alpha, \beta$ and $\gamma$ in $\Z[\sqrt{-2}]$ to be in $\mathfrak{M}(P)$ for the recurrence relation
\[
B_{\alpha+\beta}B_{\alpha-\beta}B_{\gamma}^2=B_{\alpha+\gamma}B_{\alpha-\gamma}B_{\beta}^2-B_{\beta+\gamma}B_{\beta-\gamma}B_{\alpha}^2
\] 
to hold. 

The merit of this weakened version of the recurrence relation (\ref{recurrence}) is that it allows one to work with the genuine denominator, which has better $p$--adic properties and can be applied to counting problems on elliptic curves. For example, the authors of \cite{Brauer} have used Verzobio's result \cite[Theorem~1.9]{Verzobio} to show that for a rank $1$ elliptic curve $E/\Q$ given by an integral Weierstrass equation, there is a nontrivial upper bound for the number of points whose $y$-coordinate denominator is a sum of two squares (see \cite[Theorem~1.5]{Brauer}). We believe that by using the more general result in \Cref{general CM recurrence}, one may be able to extend this to rank $2$ elliptic curves with complex multiplication, which the author is currently investigating.

\subsection{Main methodology and outline of the paper}
The original proof of \Cref{asymptotic} by Cheon and Hahn \cite{Cheon} relies heavily on the properties of the division polynomials \( \psi_n \), such as their degrees and recursive structure. Since elliptic nets are defined by rational functions rather than polynomials, their strategy does not carry over directly. Instead, our approach is inspired by Naskręcki and Verzobio \cite{Verzobio2} and Panda \cite{Panda}. In particular, we use the Néron local height function
\[
  \lambda \colon E(K)\setminus\{O\} \to \mathbb{R}
\]
to compute the local height of the point \([\alpha]P\) for \(\alpha \in \End(E)=\mathbb{Z}[\omega]\). Once we obtain an explicit formula for this (see \Cref{functional eq}), we can reconstruct the intermediate steps of \cite{Cheon} in the CM setting and thereby establish the formula in \Cref{CM gcd}. Compared with the original proof of Cheon and Hahn, our use of the Néron local height function not only removes the dependence on division polynomials, thus making the argument more flexible and applicable to other problems involving elliptic curves or elliptic nets, but also avoids heavy machinery, resulting in a cleaner and more accessible proof.

However, working with the Néron local height function also introduces certain technical difficulties. Although the expression
\[
  \lambda([a]P + [b]Q)
\]
is expected to be a quadratic function in \(a\) and \(b\), based on the known formula for \(\lambda([m]P)\), one must still determine the coefficients of \(a^2\), \(b^2\), and \(ab\) explicitly. In particular, the coefficient of the \(ab\) term is not obvious. Furthermore, \Cref{CM gcd} extends the result of Cheon and Hahn, so the main challenge lies in identifying the additional terms needed to accommodate the CM setting. As illustrated in \Cref{Bad multiple}, repeated application of the local height formula in \Cref{functional eq} produces numerous local height values which are difficult to trace where they come from; these ultimately arise from the new terms associated with CM points.

With regards to this, we outline the paper as follows: in Section 2, we review necessary facts about elliptic nets and the N\'{e}ron local height function. Section 3 then provides the formula\textcolor{red}{e} for the local height value of a linear combination of points, showing how it is related to elliptic nets. In Section 4, we will prove \Cref{CM gcd} by repeated use of the formula in Section 3. Finally, we will also follow Verzobio's strategy to show \Cref{CM recurrence} \textcolor{red}{and \Cref{general CM recurrence}} as an application in Section 5. 

\subsection{Notation}\label{Notations}
Throughout the paper, unless otherwise specified, we will use the following notations.
\begin{table}[h!]
\centering
\begin{tabular}{cc}
$K$ and $L$           & Arbitrary number fields       \\
$F$           & A quadratic imaginary field             \\
$\OO_{K}$     & The ring of integers of the field $K$   \\
$\Z[\omega]$   & An order of $F$\\
$\nu$	&               A normalised discrete valuation\\
$u$ and $w$ & Non--Archimedean valuations          
\end{tabular}
\end{table}

Therefore, our convention for this paper is as follows: let $E/K$ be an elliptic curve defined over a number field $K$ with complex multiplication by $\Z[\omega]$ (i.e. $\End(E)=\Z[\omega]$), and let its (not necessarily minimal) Weierstrass equation be given by the following:
\begin{equation}\label{EC}
E\colon y^2 + a_1xy + a_3y = x^3 + a_2x^2 + a_4x + a_6; a_i \in \OO_{K}.
\end{equation}

\section{Preliminaries}
In this section, we will collect some basic facts about the main ingredients of our proofs, consisting of orders of quadratic fields, elliptic nets and the local N\'{e}ron height function, so the paper is self-contained.  

\subsection{Orders in quadratic fields}

An order $\OO$ in a quadratic field $K$ is a subset $\OO \subset K$ such that 
$\OO$ is a subring of $K$ containing $1$ and a free $\Z$--module of rank $2$. 
For our purpose, we would like to describe the structure of $\OO$. 
First note that for a quadratic field $K$ of discriminant $d_{K}$, 
the ring of integers $\OO_{K}$ can be written as 
\[
\OO_{K} = \Z[w_{K}], \quad w_{K} = \frac{d_{K} + \sqrt{d_{K}}}{2}.
\]
The following result is taken from \cite[\S7A, page 120, Lemma~7.2]{Cox}.

\begin{lemma}\label{conductor}
Let $\OO$ be an order in a quadratic field $K$ of discriminant $d_{K}$. 
Then $\OO$ has finite index in $\OO_{K}$, and if we set the conductor of $\OO$ to be $f = \left[\OO_{K} : \OO\right]$, then 
\[
\OO = \Z + f\OO_{K} = \Z[fw_{K}].
\]
\end{lemma}

\begin{cor}\label{order structure}
Every order $\OO$ in a quadratic field $K$ of conductor $f$ is of the form $\Z[\omega]$, 
where either 
\[
\omega^2 = D, \quad \text{or} \quad \omega^2 - f\omega = D
\]
for some integer $D$.
\end{cor}

\begin{proof}
Let $K = \Q(\sqrt{N})$ with $N$ a square--free integer. 
Recall that one can describe the ring of integers as 
\[
\OO_K =
\begin{cases}
  \Z[\sqrt{N}], & \text{if } N \not\equiv 1 \bmod 4, \\[1em]
  \Z\!\left[\dfrac{1 + \sqrt{N}}{2}\right], & \text{if } N \equiv 1 \bmod 4.
\end{cases}
\]
Let $\{1, \omega\}$ be a $\Z$--basis of $\OO$. 
If $N \not\equiv 1 \pmod{4}$, then \Cref{conductor} implies that 
$\omega = f\sqrt{N}$, which has minimal polynomial 
$m_1(x) = x^2 - f^2 N$. 
If $N \equiv 1 \pmod{4}$, then \Cref{conductor} implies that 
$\omega = \left(\dfrac{1 + \sqrt{N}}{2}\right)f$, which has minimal polynomial 
\[
m_2(x) = x^2 - f x + \dfrac{f^2(1 - N)}{4}.
\]
\qedhere
\end{proof}

\subsection{Elliptic divisibility sequence and elliptic net}
An \emph{elliptic divisibility sequence} (EDS) is a sequence $(W_n)$ that takes values in an integral domain and satisfies the recursion 
\[
W_{n+m}W_{n-m}W_{r}^2=W_{n+r}W_{n-r}W_{m}^2-W_{m+r}W_{m-r}W_{n}^2
\]
for any integers $n, m, r$. Therefore, one can see that $\{\psi_{n}(P)\}_{n \in \N}$ is also an EDS. In \cite{Stange}, Stange generalised the concept of an elliptic sequence to an $n$--dimension array, also known as an \emph{elliptic net}, which extends the recurrence relation above to `a higher rank'. 

\begin{defn}[Elliptic net]\label{elliptic net}
Let $A$ be a free finitely-generated abelian group and $R$ be an integral domain. An \emph{elliptic net} is any map $W \colon A \rightarrow R$ with $W(\bm{0})=0$ and, for any $\bm{p}, \bm{q}, \bm{r}, \bm{s} \in A$,
\begin{multline*}
W(\bm{p}+\bm{q}+\bm{s})W(\bm{p}-\bm{q})W(\bm{r}+\bm{s})W(\bm{r}) \\
+ W(\bm{q}+\bm{r}+\bm{s})W(\bm{q}-\bm{r})W(\bm{p}+\bm{s})W(\bm{p}) \\ 
+ W(\bm{r}+\bm{p}+\bm{s})W(\bm{r}-\bm{p})W(\bm{q}+\bm{s})W(\bm{q})=0.
\end{multline*} \label{elliptic net def}
We call the rank of \(W\) the rank of the elliptic net. 
\end{defn}

This is indeed a generalisation of EDS: note that for an elliptic net $W: A \rightarrow R$, if we take $A=\Z$, $\bm{p}=m, \bm{q}=n, \bm{r}=r, \bm{s}=0$, then $W$ is an EDS by definition since $W$ is an \emph{odd} function.

\begin{lemma}\cite[Proposition 2.1]{Stange}
Let $A$ be a free abelian group of finite rank, and $R$ be an integral domain. Let $W\colon A \rightarrow R$ be an elliptic net. Then for any $\bm{v} \in A, W(-\bm{v})=-W(\bm{v})$. In other words, an elliptic net is an odd function. 
\end{lemma}



Just like division polynomials, we can relate elliptic nets to elliptic curves. In \cite{Stange}, Stange has shown that there is a collection of rational functions on elliptic curves over arbitrary fields that satisfy the definition of elliptic nets. In fact, they describe the collection of denominators of linear combinations of points, $\sum_{i=1}^{n}[m_{i}]P_{i}$, on an elliptic curve, which means we are taking $A=\Z^{n}$ in the definition of an elliptic net.

\begin{defn}[Net polynomial and elliptic denominator net]
For an arbitrary field $K$, consider the affine coordinate ring 
\[
R_{n}=K[x_i, y_i]_{1\leq i \leq n}[(x_{i}-x_{j})^{-1}]/\langle f(x_{i}, y_{i})\rangle_{1\leq i \leq n},
\]
where $f$ is the defining polynomial (\ref{EC}) for the elliptic curve $E/K$. Let $\bm{P}=(P_{1}, ..., P_{n}) \in E(K)^{n}$ and $\bm{v}=(v_{1}, \dotsc, v_{n}) \in \Z^{n}$. If $P_{i}\neq O$ for all $i$ and $P_i \neq P_j$ for any $i \neq j$, then there exists rational functions $\Psi_{\bm{v}}(\bm{P})$, $\Phi_{\bm{v}}(\bm{P})$, $\bar\Omega_{\bm{v}}(\bm{P}) \in R_{n}$ such that 
\begin{equation}\label{notation}
\bm{v} \cdot \bm{P} = v_{1}P_{1} +...+v_{n}P_{n}=\left(\frac{\Phi_{\bm{v}}(\bm{P})}{\Psi^{2}_{\bm{v}}(\bm{P})}, \frac{\bar\Omega_{\bm{v}}(\bm{P})}{\Psi^{3}_{\bm{v}}(\bm{P})}\right)=\left(\frac{A_{\bm{v} \cdot \bm{P}}}{D^{2}_{\bm{v} \cdot \bm{P}}}, \frac{B_{\bm{v} \cdot \bm{P}}}{D^{3}_{\bm{v} \cdot \bm{P}}}\right) .
\end{equation}

\begin{enumerate}
\item The polynomial $\Psi_{\bm{v}}$ is defined to be the \emph{$\bm{v}$-th net polynomial}, which is an elliptic net (for proof, see \cite[Theorem 4.1]{Stange}).
\item The element $D_{\bm{v} \cdot \bm{P}}$ is called the \emph{elliptic denominator net}. 
\end{enumerate}
\end{defn}

The following properties of net polynomials will be useful to take note of in this paper. We refer the readers to \cite[Lemma 2.5, Lemma 2.6, Theorem 2.8]{Stange} and \cite[Lemma 2.5]{Akbary} for the proof. 
\begin{prop}[Properties of net polynomials]\label{net poly}
The net polynomials satisfy the following properties:
\begin{enumerate}[(a)]
\item $\Psi_{\bm{v}}$ is an elliptic net, so the recurrence in \Cref{elliptic net} holds;
\item For all $\bm{v} \in \mathbb{Z}^{n}$, the values of the $\bm{v}$-th net polynomial are completely defined by the following initial conditions:
\begin{itemize}
\item $\Psi_{\bm{0}} = 0$;
\item $\Psi_{\bm{e}_i} = 1$ for $1\leq i \leq r$;
\item $\Psi_{2\bm{e}_i} = 2y_i + a_{1}x_i + a_{3} = \psi_{2}(P_i)$;
\item $\Psi_{\bm{e}_i + \bm{e}_j} = 1$, $i \neq j$;
\item $\Psi_{2\bm{e}_i + \bm{e}_j} = 2x_i + x_j - \left(\frac{y_j - y_i}{x_j - x_i}\right)^{2}
      - a_{1}\left(\frac{y_j - y_i}{x_j - x_i}\right) + a_{2}$, $i \neq j$.
\end{itemize}
\label{a values}
\item For $1\leq i \leq r$, we have $\Psi_{n\bm{e}_{i}}(\bm{P})=\psi_{n}(P_i)$.
\item For any $\bm{v}, \bm{u} \in \mathbb{Z}^{r}$,
\begin{equation}
\Psi_{\bm{v}}^{2}\Psi_{\bm{u}}^{2}\left(x\left(\bm{v} \cdot \bm{P}\right)
- x\left(\bm{u} \cdot \bm{P}\right)\right)
= -\Psi_{\bm{v} + \bm{u}}\Psi_{\bm{v} - \bm{u}},
\end{equation}
\label{notation}
where $x(\bm{v} \cdot \bm{P})$ represents the $x$-coordinate of the point $\bm{v} \cdot \bm{P}$.

\item Using the notations in \ref{notation}, for $1 \leq i \leq r$, substitute $\bm{u} = \bm{e}_i$ in \ref{notation}; then we have
\begin{equation}
\Phi_{\bm{v}}(\bm{P})
= \Psi_{\bm{v}}^{2}(\bm{P})\,x(P_i)
- \Psi_{\bm{v} + \bm{e}_i}(\bm{P})\Psi_{\bm{v} - \bm{e}_i}(\bm{P}).
\end{equation}
\label{net poly b}
\end{enumerate}
\end{prop}
When we say net polynomials are defined completely by the initial conditions given in \Cref{net poly}\ref{a values}, we mean by choosing appropriate initial values and substituting them into the definition of an elliptic net, one can obtain some other values, and hence define all values recursively. One can refer to \cite[Theorem 2.5]{Stange} for examples. For our purpose, now we will use the following notations: for $\alpha = a+b\omega \in \End(E)=\Z[\omega]$, define
\begin{equation}\label{nett}
\psi_{\alpha}(P)\coloneqq \Psi_{\bm{\alpha}}\left(\bm{P}\right), \,\, \text{and}\,\, \phi_{\alpha}(P) \coloneqq \Phi_{\bm{\alpha}}(\bm{P})
\end{equation}
where $\bm{\alpha}=(a, b) \in \Z^2$ represents the coordinate form of $\alpha$. Now as in the case for EDS, for $\alpha=\alpha_{1}+\alpha_{2}\omega, \beta=\beta_{1}+\beta_{2}\omega, \gamma=\gamma_{1}+\gamma_{2}\omega$, if we take $\bm{p}=(\beta_{1}, \beta_{2}), \bm{q}=(\alpha_{1}, \alpha_{2}), \bm{r}=(\gamma_{1}, \gamma_{2}), \bm{s}=(0, 0)$ in (\ref{elliptic net def}), then we get 
\[
\psi_{\alpha+\beta}\psi_{\alpha-\beta}\psi_{\gamma}^{2}=\psi_{\alpha+\gamma}\psi_{\alpha-\gamma}\psi_{\beta}^{2}-\psi_{\beta+\gamma}\psi_{\beta-\gamma}\psi_{\alpha}^{2},
\]
which justifies the use of elliptic nets to be our CM division polynomials. 

The final fact about elliptic nets we need for this paper is the equivalence relation for elliptic nets, namely \emph{scale equivalence}. This is done by relating them to certain `quadratic forms'. 

\begin{defn}
Let $B$ and $C$ be abelian groups written additively. A function $F \colon B \rightarrow C$ is a quadratic form if for all $x, y, z \in B$, 
\begin{equation}\label{hi}
F(x+y+z)+F(x)+F(y)+F(z)=F(x+y)+F(x+z)+F(y+z).
\end{equation}
If $F \colon B \rightarrow C$ is a quadratic form, then it satisfies the parallelogram law, that is, for all $x, y \in B$,
\begin{equation}\label{bye}
F(x+y)+F(x-y)=2F(x)+2F(y).
\end{equation}
Furthermore, the converse holds if $C$ is 2-torsion free. 
\end{defn}

\begin{prop}
Let $K$ be a field and let $W\colon A \rightarrow K$ be an elliptic net. Let $F\colon A \rightarrow K^{\times}$ be a quadratic form. Then the function $W^{F} \colon A \rightarrow K$ defined by
\[
W^{F}(\bm{v})=W(\bm{v})F(\bm{v})
\]
is also an elliptic net. 
\end{prop}

\begin{proof}
We just need to show that $F(\bm{v})$ also satisfies the recurrence relation of elliptic nets. To do so, \textcolor{red}{first write (\ref{hi}) and (\ref{bye}) in multiplicative form and multiply them together. We then utilise the symmetry in quadratic forms. We omit the computation here; interested readers may} refer to \cite[Proposition 6.1]{Stange}.
\end{proof}

\begin{defn}
We say two elliptic nets $W_{1}$ and $W_{2}$ are scale equivalent if there exists a quadratic form $F\colon A \rightarrow K$ such that $W_{1}=W_{2}^{F}$. 
\end{defn}

\subsection{Local N\'{e}ron height function}
Let $\lambda : E(K)\backslash\{O\} \to \mathbb{R}$ be the local N\'{e}ron height function as defined in \cite[Chapter VI, page 455, Theorem 1.1]{Silverman2}. We recall some classical facts on the local N\'{e}ron height, which we took from \cite[Chapter VI, page 476-7, Exercise 6.3, 6.4(e)]{Silverman2}.
\begin{prop}
Let \( K \) be a field complete with respect to a non-Archimedean valuation \( u \)\textcolor{red}{, and let} $E/K$ be an elliptic curve\textcolor{red}{. Then} the following hold:
\begin{enumerate}
\item For all $P, Q \in E(K)$ with $P, Q, P\pm Q \neq O$, \textcolor{red}{the N\'{e}ron height function $\lambda\colon E(K)\setminus \{O\} \rightarrow \R$ associated to the non-Archimedean valuation $u$} satisfies the quasi-parallelogram law
\[
\lambda(P+Q)+\lambda(P-Q)=2\lambda(P)+2\lambda(Q)+u(x(P)-x(Q))-\frac{1}{6}u(\Delta).
\]
\item The local N\'{e}ron height function does not change if we replace $K$ with a finite extension. Moreoever, $\lambda(\cdot)$ does not depend on the choice of Weierstrass equation for $E/K$.
\end{enumerate}
\end{prop}

\begin{cor}\label{functional}
For any non $m$-torsion point $P$, with $m \in \Z$,
\[
\lambda([m]P)=m^{2}\lambda(P)+u(\psi_{m}(P))-\frac{1}{12}(m^{2}-1)u(\Delta).
\]
\end{cor}

The above results demonstrate how the valuations of $\psi_m$ are connected to local N\'{e}ron heights, allowing us to avoid addressing any properties of division polynomials. To make it look neater without the valuation of the discriminant, we can normalise the local height funciton as in \cite[\emph{Second normalisation}, page 90]{Serre}. 

\begin{defn}
The normalised local N\'{e}ron height function is defined to be 
\[
\widetilde{\lambda}(P)\coloneqq \lambda(P)-\frac{1}{12}u(\Delta).
\]
\end{defn}

Therefore, the `normalised' quasi-parallelogram law and corollary above are respectively,
\begin{equation}\label{tt}
\widetilde{\lambda}(P+Q)+\widetilde{\lambda}(P-Q)=2\widetilde{\lambda}(P)+2\widetilde{\lambda}(Q)+u(x(P)-x(Q))
\end{equation}
and 
\[
\widetilde{\lambda}([m]P)=m^{2}\widetilde\lambda(P)+u(\psi_{m}(P)).
\]

One of the reasons why we choose this normalisation is because when a point $P$ is not singular $\bmod u$, then the N\'{e}ron local height function has a nice formula, which we take from \cite[Chapter VI, page 470, Theorem 4.1]{Silverman2}.

\begin{theorem}\label{formula}
Let \( K \) be a field complete with respect to a non-Archimedean valuation \( u \), let \( E/K \) be an elliptic curve, and choose a Weierstrass equation for \( E \) with \( u \)-integral coefficients,
\[
E\colon y^2 + a_1 xy + a_3 y = x^3 + a_2 x^2 + a_4 x + a_6.
\]

Let \( \Delta \) be the discriminant of this equation. \textcolor{red}{If $P$ has good reduction modulo $u$,} then the Néron local height function \( \lambda : E(K) \setminus \{ O \} \to \mathbb{R} \) is given by the formula
\begin{equation}\label{good reduction local height}
\lambda(P) = \frac{1}{2} \max \left(u(x(P)^{-1}), 0 \right) + \frac{1}{12} u(\Delta),
\end{equation}
and the normalised version is 
\[
\widetilde{\lambda}(P) = \frac{1}{2} \max (u(x(P)^{-1}), 0 )
\]
\end{theorem}

The local height is simply a function on the points of $E$ and does not depend on $m$ in any way, so its use is not related to whether the elliptic curve has complex multiplication or not. However, \Cref{functional} only considers integral multiples of a point, therefore we will extend this result to cover the case of complex multiplication as well.

\subsection{Elliptic curves with complex multiplication}\label{module}
\textcolor{red}{Let $\nu$ be a finite place associated to an $\OO_K$ prime $\p$ and $K_{\nu}$ be the completion of $K$ with respect to $\nu$. We} denote by \textcolor{red}{$E_0(K_{\nu})$} the set of points of $E(K_{\nu})$ with non--singular reduction modulo $\p$. From \cite[\S7, Proposition VII.6.3]{Silverman1}, we know that for $\textcolor{red}{n \geq 1}$, there is a subgroup filtration
\begin{equation}\label{filtration}
\cdots \subset E_{2}(K_{\textcolor{red}{\nu}}) \subset E_{1}(K_{\textcolor{red}{\nu}}) \subset E_0(K_{\textcolor{red}{\nu}}), \, E_{n}(K_{\textcolor{red}{\nu}})=\{P\in \textcolor{red}{E_{0}(K_\nu)}\colon \nu(x(P))\leq -2n\} \cup \{O\}.
\end{equation}
When $E$ has complex multiplication by $\Z[\omega]$, then in fact the above filtration forms a $\Z[\omega]$--module filtration. This was proven by Streng in \cite[Corollary 2.10]{Streng} using formal groups and the theory of N\'{e}ron models. 
\begin{theorem}\cite[Corollary 2.10]{Streng}\label{Streng module}
For any $n\geq 1$, the group \textcolor{red}{$E_{n}(K_\nu)$} is a $\Z[\omega]$--submodule of \textcolor{red}{$E(K_\nu)$}. Moreover, we have an isomorphism of $\Z[\omega]$--modules
\[
E_{n}(K_{\textcolor{red}{\nu}})/E_{n+1}(K_{\textcolor{red}{\nu}})\cong l, 
\]
where $l=\OO_{K_{\textcolor{red}{\nu}}}/\p$ is the residue field of $K_{\textcolor{red}{\nu}}$. 
\end{theorem}

For our purpose, we further need the group $E(K_{\textcolor{red}{\nu}})/E_{0}(K_{\textcolor{red}{\nu}})$ to be a $\Z[\omega]$--module since we are trying to extend the explicit formula by Cheon and Hahn. In later sections, we will often talk about the annihilator of $P$ in the $\Z[\omega]$--module $E(K_{\textcolor{red}{\nu}})/E_{0}(K_{\textcolor{red}{\nu}})$.

Under the assumption that the elliptic curve \( E/K \) has complex multiplication by \( \mathbb{Z}[\omega] \), \( E \) has an integral \( j \)-invariant (see \cite[Theorem 6.1, p. 140]{Silverman2}), and therefore it has potential good reduction (see \cite[Proposition 5.5, p. 197]{Silverman1}). In other words, there exists a field extension \( L/K \) such that \( E/L \) has good reduction. Since the local height function is invariant under field extensions, for a singular point \( P \in E(K) \), we only need to compute \( \lambda_L(P) \) using \Cref{formula}. The only complication is that, in the case of good reduction, the formula in \Cref{formula} depends on the \( x \)-coordinate of the point \( P \). When we extend the field for the elliptic curve, the Weierstrass equation may no longer be minimal, and hence we must perform a change of coordinates before reduction modulo $\p$.

\begin{remark}
\textcolor{red}{In the proof of \cite[Lemma 25]{Panda}, the author stated that if we have additive reduction, then we can consider a field extension where the curve has good reduction. The argument for this case is not fully-explained in \cite{Panda}, therefore the argument is elaborated here.} 
\end{remark}

Now let $E/K$ be an elliptic curve with potential good reduction $\Mod \nu$, with $\nu$ being a finite place in $K$, and let
\begin{equation}\label{Weierstrass}
E\colon y^2+a_1xy+a_3y=x^3+a_2x^2+a_4x+a_6
\end{equation}
be the Weierstrass equation for $E/K$. We will construct a finite extension $L/K$ such that $E/L$ has good reduction, with $w$ the finite place that extends $\nu$ in $L$. Now we use the formula
\begin{equation}\label{chg of coord}
x_{K}=u^2x_{L}+\gamma;\quad y_{K}=u^3y_{L}+u^2sx_{L}+t
\end{equation}
to give a $w$--minimal Weierstrass equation for $E/L$, where $u, s, t, \gamma \in L_{w}$, $(x_{K}, y_{K})$ and $(x_{L}, y_{L})$ are coordinates of a point $P$ before and after the change of coordinates respectively. In particular, $w(u)\geq 1$ since the original Weierstrass equation is not minimal. 

\begin{remark}
In \S5, where we give the proof of \Cref{CM recurrence}, we will use the valuation $\nu(\cdot)$ as an exponent later; $\nu$ therefore has to be a normalised discrete valuation. Given our purpose, as an extension of the finite place $\nu$, the finite place $w$ here is not a normalised discrete valuation (so the valuation group of $w$ need not be the integers), but just a non-Archimedean valuation.
\end{remark}

\begin{lemma} \label{reduction lemma}
Using the notation in (\ref{chg of coord}), if \( P = (x_K, y_K) \Mod \nu$ is singular with additive reduction, then \( (x_K, y_K) \equiv (\gamma, t) \Mod w \).  
\end{lemma}

\begin{proof}
We first assume that \( w \) and \( \nu \) are finite places not associated with \( 2 \). Note that by performing the linear transformation $(x, y)\mapsto \left(x, \frac{1}{2}(y-a_{1}x-a_3)\right)$, we can always transform the Weierstrass equation of $E/K$ into the form
\begin{equation}\label{transform1}
E\colon y^2=4x^3+b_2^2+2b_4x+b_6.
\end{equation}
Since the finite places are assumed not to be associated with $2$, this does not affect the type of reduction $\Mod \nu$. Therefore, without loss of generality we will assume $E/K$ has a Weierstrass equation of the form (\ref{transform1}). 

Following the proof of \cite[Chapter III.1, Proposition 1.7(a), page 49]{Silverman1}, we know that if we take the substitution $(x', y')=(x, 2y)$ then we can factorise the cubic equation over $\bar{\Q}$ to get the equation
\[
y'^2=(x'-e_1)(x'-e_2)(x'-e_3),
\]
with $e_{i}$'s are seen to be distinct. Define $L=K\left(e_1, (e_2-e_1)^{\frac{3}{2}}\right)$, and make the substitution 
\[
x_{K}'=(e_2-e_1)x_{L}+e_{1}; \quad y_{K}'=(e_{2}-e_{1})^{\frac{3}{2}}y_{L}
\]
to give the elliptic curve $E/L$ in Legendre form 
\[
E\colon y^2=x(x-1)(x-\lambda), \lambda =\frac{e_{3}-e_{1}}{e_{2}-e_{1}} \in L\setminus \{0, 1\}.
\]

We claim that this last equation has coefficients which are $w$--adic integers, and that it has good reduction modulo $w$, so this is indeed a $w$--minimal Weierstrass equation for $E/L$. Indeed, as split multiplicative reduction is invariant under field extensions, we see from the proof of  \cite[Chapter VII.5, Proposition 5.4(c), page 197]{Silverman1} that we must have $w(\lambda) \geq 0$ and good reduction. By assumption, $P$ is a cusp point $\Mod \nu$ and therefore $e_1\equiv e_2 \equiv e_3 \mod w$, and in particular $e_2-e_1 \equiv 0 \mod w$. From (\ref{transform1}), we then have $(x_K', y_K') \equiv (e_1, 0) \mod w$. 

In the case where \(w\) and \(v\) are associated to \(2\), we use the Deuring normal form and the proof is an elaboration of \cite[Appendix A, Proposition 1.2(b), 1.3, page 410--412]{Silverman1}. Given a general Weierstrass equation (\ref{Weierstrass}), in characteristic 2, the $j$--invariant is given by $j=\frac{a_{1}}{\Delta}$. Since we are assuming $E/K$ is singular under reduction modulo $\nu$ and $j$ is integral (so $j$ is not infinity modulo $\nu$), we have $\Delta \equiv 0 \Mod \nu$ and this implies we must have $a_{1} \equiv 0 \Mod \nu$. According to \cite[Appendix A, Proposition 1.1(c)]{Silverman1}, we use the transformation 
\[
x=x'+a_{2}; \quad y=y'
\]
to give 
\begin{equation}\label{transform}
E'/K\colon y'^2+a_{1}x'y'+(a_1a_2+a_3)y'=x'^3+4a_2x'^2+(5a_{2}^2+a_4)x'+(2a_{2}^3+a_{4}a_{2}+a_{6}),
\end{equation}
which has the form 
\[
y^2+b_3y=x^3+b_4x+b_6
\]
under reduction modulo $\nu$, with $\Delta = b_3^4\equiv(a_1a_2+a_3)^4 \Mod \nu$. This implies 
\[
a_1a_2+a_3 \equiv a_3^4 \equiv b_3^4 \equiv 0 \Mod \nu.
\]

Since $E'/K$ has potential good reduction, we can transform the reduction of $E' \bmod \nu$ into the Deuring Normal form 
\begin{equation}\label{Deuring}
E_{\alpha}/\F_{w}\colon y^2+\alpha xy +y =x^3, \F_{w}\,\, \text{an extension of the finite field} \,\, \F_{\nu},
\end{equation}
which has good reduction everywhere, modulo $\nu$ in particular. To do so, we will need to find an extension $L/K$ such that $w$ is a prime of $L$, and a transformation 
\[
x'_{K}=x_{L}u^2+\gamma; \quad y'_{K}=y_{L}u^3+u^2sx_{L}+t
\]
such that $E_{\alpha}$ is the reduction of $E/L \bmod w$. Apply the above transformation to (\ref{transform}), we get 
\begin{align*}
E/L\colon y_{L}^2 &+ u^{-1}(2s+a_1)x_{L}y_{L} 
+ u^{-3}(a_1a_2+a_{1}\gamma+a_3+2t)y_{L} \\
&= x_{L}^3 + u^{-2}(3\gamma-4a_{2}-a_1s-s^2)x_{L}^{2} \\
&\quad + u^{-4}(3\gamma^2-a_1a_2s-a_1\gamma s-5a_2^2-8a_2\gamma-a_3s-a_1t-2st-a_4)x_{L} \\
&\quad + (\gamma^3-2a_2^3-5a_2^2\gamma-4a_2\gamma^2-a_1a_2t-a_1\gamma t-a_2a_4-a_4\gamma-a_3t-t^2-a_6).
\end{align*}
Reduction modulo $w$ will give the Deuring Normal form. Recall that $\nu, w$ are valuations associated to $2$ and $a_1\equiv 0 \bmod \nu$, hence $a_1\equiv 0 \bmod w$. Under reduction modulo $w$, $E/L$ has the form
\begin{equation}\label{reduce}
\begin{aligned}
y_{L}^2+u^{-3}a_3y_L
&\equiv x^3+u^{-2}(\gamma-s^2)x_L^{2}+u^{-4}(\gamma^2-a_2^2-a_3s-a_4)x_L \\
&\quad+(\gamma^3-a_2^2\gamma-a_2a_4-a_4\gamma-a_3t-t^2-a_6).
\end{aligned}
\end{equation}
After reduction modulo $w$, (\ref{reduce}) is in Deuring form, by comparing coefficients of (\ref{reduce}) and (\ref{Deuring}), we found the following relationships over $L$ modulo $w$:
\[
a_3u^{-3}\equiv 0, \alpha \equiv 0, \gamma\equiv s^2, \gamma^2\equiv a_4+a_2^2+a_3s, \gamma^3-a_2^2\gamma-a_4\gamma-a_4a_2-a_3t-a_6-t^2\equiv 0.
\]
Note that we always have $u\equiv 0 \bmod w$, so the first relation implies $a_3 \equiv 0 \bmod w$ necessarily and hence we have 
\[
\gamma^2\equiv a_4+a_2^2, t^2\equiv a_4a_2+a_6 \pmod w.
\] 
On the other hand, if we compute the partial derivatives of (\ref{transform}) to find the coordinates of the singular point mod $\nu$, we require 
\[
\frac{\partial}{\partial y} \equiv a_3\equiv 0, \frac{\partial}{\partial x}\equiv -3x'^2+a_4+a_2^2\equiv 0 \pmod \nu.
\]
This implies the $x$--coordinate of the singular point satisfies $(x'_{K})^2 \equiv a_{4}+a_{2}^2 \equiv \gamma^{2} \bmod w$. This implies $(x'_{K})^{2}-\gamma^{2} \equiv (x'_{K}-\gamma)^2 \equiv 0 \bmod w$, hence we have $x'_{K}\equiv \gamma \bmod w$.

Now reduce $E'/L$ reduction modulo $w$, using the identities above we obtain 
\begin{align*}
(y'_K)^2
&\equiv (x')_K^3+(a_2^2+a_4)x'_{K}+(a_4a_2+a_6) \\
&\equiv 2(x')_K^2+t^2 \\
&\equiv t^2	\pmod w.
\end{align*}
This implies $(y'_{K})^{2}-t^{2} \equiv (y'_{K}-t)^2 \equiv 0 \bmod w$, hence we have $y'_{K}\equiv t \bmod w$, as we required.
\end{proof} 

\begin{theorem}
The group $E_{0}(K)$ is a $\Z[\omega]$--submodule of $E(K)$.
\end{theorem}
\begin{proof}
Let $P = (x_K, y_K) \in E_0(K)$. Using the notations in (\ref{chg of coord}), we know that $x_K \not\equiv \gamma \Mod{w}$ by \Cref{reduction lemma}, so we must have $w\left(x_L\right) \leq -2w\left(u\right)$. This implies that there is an injection map $E_0(K) \xhookrightarrow{} E_{i_0}(L)$, where $i_0 = w(u) \geq 1$ and for $n\geq 1$, we also have a filtration
\begin{equation}\label{w filtration}
\cdots E_{3}(L_{w}) \subset E_{2}(L_{w}) \subset E_{1}(L_{w}) \subset E_0(L_{w}),
\end{equation}
where $E_{n}(L_{w})=\{P\in E(L_{w})\colon w(x_L(P))\leq -2nw(\pi_w)\} \cup \{O\}$ and $\pi_w$ is a uniformiser for $w$. Also note that if $P=(x_{L}, y_L) \in E_{i_0}(L)$, then by definition $w(x_L) \leq -2w(u) = -2i_0$. From (\ref{chg of coord}), this tells us that $w(x_K-\gamma)\leq 0$, which is equivalent to $x_K \not\equiv \gamma \bmod w$. Overall, we have $E_0(K) = E(K) \cap E_{i_0}(L)$. 

From \Cref{Streng module}, we know that both sets on the right-hand side are closed under the $[\omega]$ map. Therefore, $E_0(K)$ is also closed under the $[\omega]$ map, which proves that it is a $\mathbb{Z}[\omega]$-submodule of $E(K_\nu)$.
\end{proof}

\textcolor{red}{Now that we have shown $E(K_{\nu})/E_{0}(K_{\nu})$ to be a $\Z[\omega]$--module, we make the following definition.}

\begin{defn}\label{ANN}
\textcolor{red}{
Let $E$ be an elliptic curve defined by a Weierstrass equation with coefficients in $\OO_K$ and $\nu$ be a normalised discrete valuation with respect to a prime ideal $\p$ of $\OO_K$. For $P \in E(K)$, denote $\A(\p, P)$ to be the annihilator of $P$ in the $\Z[\omega]$--module $E(K_{\nu})/E_{0}(K_{\nu})$.
}
\end{defn}
\section{Local height values of CM points}
\subsection{Extended formula for local height function values}
When we replace $m \in \Z$ with $\alpha=a+b\omega \in \Z[\omega]$ in \Cref{functional}, a natural guess is to replace the square function $m \mapsto m^2$ with the norm function $\alpha \mapsto \text{Norm}(\alpha)$. However, note that net polynomials treat $P$ and $[\omega]P$ as two different points on $E(K)$, so $[a+b\omega]P$ is interpreted as applying two endomorphisms $[a]$ and $[b]$ to $P$ and $[\omega]P$ respectively, rather than applying the single endomorphism $[a+b\omega]$ to $P$ only. Therefore, the local height value \textcolor{red}{remains} a quadratic function in $a$ and $b$ with coefficients being the normalised local heights.

\begin{prop}\label{functional eq}
Let $E$ be an elliptic curve defined over a number field $K$, $\bm{v}=(a, b) \in \Z^{2}$. Let $\bm{P}=(P, Q) \in E(K)^{2}$ \textcolor{red}{and assume $\bm{v} \cdot \bm{P} \neq O$}. If $\nu$ is the finite place associated to the prime $\p$, then the normalised N\'{e}ron local height function satisfies
\[
\widetilde{\lambda}(\bm{v} \cdot \bm{P})=a^{2}\widetilde{\lambda}(P)+b^{2}\widetilde{\lambda}(Q)+ab\left(\widetilde{\lambda}(P+Q)-\widetilde{\lambda}(P)-\widetilde{\lambda}(Q)\right)+\nu\left(\Psi_{(a, b)}(\bm{P})\right),
\]
where $\Psi_{(a, b)}$ is the elliptic net polynomial evaluated at $\bm{v}=(a, b)$. 
\end{prop}
\begin{proof}
 We would first induct on $a \in \Z_{\geq 0}$, where our base cases are $a=0$ and $a=1$. For $a=1$, we will do subproofs of induction on $b\geq -1$ and  $b\leq 0$. Next, we will apply \textcolor{red}{the} inductive hypothesis \textcolor{red}{to} $a$ and once it is done, we use the symmetry argument $\widetilde{\lambda}(P) = \widetilde{\lambda}(-P)$ to establish the statement for $(a, b) \in \Z^2$. 

We first dispose of $(a, 0), (0, b)$ and $(a, b)=(1, -1)$. Note that if one of \(a\) or \(b\) is 0, then we can apply Corollary \ref{functional} to the point \([a]P\) or \([b]Q\) directly and the result holds. While for $(a, b)=(1, -1)$, this is just the quasi-parallelogram law: if $(a, b)=(1, -1)$, then 
\begin{align*}
\text{LHS}&=\widetilde{\lambda}(P-Q)\\
\text{RHS}
&=2\widetilde{\lambda}(P)+2 \widetilde{\lambda}(Q)-\widetilde{\lambda}(P+Q)+\nu\left(\Psi_{(1, -1)}(\bm{P})\right)\\
&=2\widetilde{\lambda}(P)+2 \widetilde{\lambda}(Q)-\widetilde{\lambda}(P+Q)+\nu\left(x(Q)-x(P)\right), 
\end{align*}
where $\Psi_{(1, -1)}(\bm{P})=x(Q)-x(P)$ was given in \cite[Proposition 3.8]{Stange}. Since $\nu\left(x(Q)-x(P)\right)=\nu\left(x(P)-x(Q)\right)$, rearrange the terms to resemble (\ref{tt}).


For the inductive step, we first induct on $a=1$ and $b\geq -1$. Fix \(b\) and suppose the statement is true for $a=1$ and each $b \in [-1, k]$ for some $k \in \N$, consider $(a, b)=(1, k+1)$. Apply the quasi-parallelogram law with \(P = P+[k]Q\) and \(Q = Q\) and so we get 
\begin{align*}
\widetilde{\lambda}(P+[k+1]Q)
&= \widetilde{\lambda}\left((P+[k]Q) + Q\right) \\
&= 2\widetilde{\lambda}\left(P+[k]Q\right) + 2\widetilde{\lambda}(Q) - \widetilde{\lambda}(P + [k-1]Q) + \nu\left(x\left(P+[k]Q\right) - x(Q)\right).
\end{align*}
Now apply the inductive hypothesis to 
\(\widetilde{\lambda}(P + [k]Q)\) and \(\widetilde{\lambda}(P + [k-1]Q)\), and then apply property~\ref{notation} in \Cref{net poly} to rewrite the expression for the valuation of the difference between the two $x$--coordinates, so that we have
\begin{align*}
\widetilde{\lambda}(P+[k+1]Q)
&= 2\widetilde{\lambda}(P) + 2k^{2}\widetilde{\lambda}(Q) + 2k\left(\widetilde{\lambda}(P+Q) - \widetilde{\lambda}(P) - \widetilde{\lambda}(Q)\right) + 2\nu\left(\Psi_{(1, k)}(\bm{P})\right) \\
&\quad + 2\widetilde{\lambda}(Q)- \widetilde{\lambda}(P)  \\
&\quad - (k-1)^{2}\widetilde{\lambda}(Q) - (k-1)\left(\widetilde{\lambda}(P+Q) - \widetilde{\lambda}(P) - \widetilde{\lambda}(Q)\right) \\
&\quad - \nu\left(\Psi_{(1, k-1)}(\bm{P})\right) + \nu\left(-\frac{\Psi_{(1, k+1)}(\bm{P})\Psi_{(1, k-1)}(\bm{P})}{\Psi_{(1, k)}^{2}(\bm{P})}\right) \\
&= \widetilde{\lambda}(P) + (k+1)^2\widetilde{\lambda}(Q) + (k+1)\left(\widetilde{\lambda}(P+Q) - \widetilde{\lambda}(P) - \widetilde{\lambda}(Q)\right) \\
&\quad + \nu\left(\Psi_{(1, k+1)}(\bm{P})\right).
\end{align*}

This shows that the statement holds for any $b \geq -1$ when $a=1$. Next, we induct on $a=1$ and $b\leq 0$. Again, fix \(b\) and suppose the statement is true for $a=1$ and each $b \in [-k, -1]$ for some $k \in \N$, consider $(a, b)=(1, -k-1)$. The proof goes similarly as above: first apply the quasi-parallelogram law with \(P = P+[-k]Q\) and \(Q = Q\), then the inductive hypothesis to \(\widetilde{\lambda}\left(P+[-k]Q\right)\) and \(\widetilde{\lambda}(P + [-k+1]Q)\) and finally apply property \ref{notation} in Proposition \ref{net poly}. We get 
\begin{align*}
\widetilde{\lambda}(P+[-k-1]Q)
&= \widetilde{\lambda}\left((P+[-k]Q) - Q\right) \\
&= 2\widetilde{\lambda}\left(P+[-k]Q\right) + 2\widetilde{\lambda}(Q) - \widetilde{\lambda}(P + [-k+1]Q) \\
&\quad + \nu\left(x\left(P+[-k]Q\right) - x(Q)\right) \\
&= 2\widetilde{\lambda}(P) + 2k^{2}\widetilde{\lambda}(Q) - 2k\left(\widetilde{\lambda}(P+Q) - \widetilde{\lambda}(P) - \widetilde{\lambda}(Q)\right)  \\
&\quad + 2\nu\left(\Psi_{(1, -k)}(\bm{P})\right) + 2\widetilde{\lambda}(Q) - \widetilde{\lambda}(P) - (-k+1)^{2}\widetilde{\lambda}(Q) \\
&\quad - (-k+1)\left(\widetilde{\lambda}(P+Q) - \widetilde{\lambda}(P) - \widetilde{\lambda}(Q)\right) \\
&\quad - \nu\left(\Psi_{(1, -k+1)}(\bm{P})\right) + \nu\left(-\frac{\Psi_{(1, -k+1)}(\bm{P})\Psi_{(1, -k-1)}(\bm{P})}{\Psi_{(1, -k)}^{2}(\bm{P})}\right) \\
&= \widetilde{\lambda}(P) + (-k-1)^2\widetilde{\lambda}(Q)  \\
&\quad - (k+1)\left(\widetilde{\lambda}(P+Q)  - \widetilde{\lambda}(P) - \widetilde{\lambda}(Q)\right) + \nu\left(\Psi_{(1, -k-1)}(\bm{P})\right).
\end{align*}

Overall, we have shown that the statement holds for $a=1$ and $b\in \Z$. Now we need to induct on $a\geq 0$. Suppose the statement holds for holds for each $a \in [-1, k]$ for some $k\in \N$, consider $(a, b)=(k+1, b)$. As before, we apply the quasi-parallelogram law with \(P = [k]P+[b]Q\) and \(Q = P\), then the inductive hypothesis to \(\widetilde{\lambda}([k]P+[b]Q\textcolor{red}{)}\) and \(\widetilde{\lambda}([k-1]P + [b]Q)\), and property \ref{notation} in \Cref{net poly}. We get 
\begin{align*}
\widetilde{\lambda}([k+1]P+[b]Q)
&= \widetilde{\lambda}\left(([k]P+[b]Q) + P\right) \\
&= 2\widetilde{\lambda}([k]P + [b]Q) + 2\widetilde{\lambda}(P) - \widetilde{\lambda}([k-1]P + [b]Q) \\
&\quad + \nu\left(x([k]P + [b]Q) - x(P)\right) \\
&= 2k^{2}\widetilde{\lambda}(P) + 2b^{2}\widetilde{\lambda}(Q) + 2kb\left(\widetilde{\lambda}(P+Q) - \widetilde{\lambda}(P) - \widetilde{\lambda}(Q)\right)  \\
&\quad + 2\nu\left(\Psi_{(k, b)}(\bm{P})\right) + 2\widetilde{\lambda}(P)- (k-1)^{2}\widetilde{\lambda}(P) - b^{2}\widetilde{\lambda}(Q)  \\
&\quad  - (k-1)b\left(\widetilde{\lambda}(P+Q) - \widetilde{\lambda}(P) - \widetilde{\lambda}(Q)\right) \\
&\quad - \nu\left(\Psi_{(k-1, b)}(\bm{P})\right) + \nu\left(-\frac{\Psi_{(k+1, b)}(\bm{P})\Psi_{(k-1, b)}(\bm{P})}{\Psi_{(k, b)}^{2}(\bm{P})}\right) \\
&= (k + 1)^{2}\widetilde{\lambda}(P) + b^{2}\widetilde{\lambda}(Q) + \nu\left(\Psi_{(k+1, b)}(\bm{P})\right)\\
&\quad \textcolor{red}{+(k+1)b\left(\widetilde{\lambda}(P+Q) - \widetilde{\lambda}(P) - \widetilde{\lambda}(Q)\right).}
\end{align*}
This shows that if the statement holds for $(a, b)$, then it holds for $(a+1, b)$ for $a\in \Z_{\geq 0}, b\in \Z$. By the symmetry argument, this also implies the statement holds for $a\in \Z_{\leq 0}$ and $b \in \Z$. Hence, the statement holds for any $(a, b) \in \Z^2$ by induction. 
\qedhere
\end{proof}
In the result above, there is no restriction about the points $P, Q\in E(K)$ chosen, so we can just apply the result to our specific setting of CM and get the corollary below. 
\begin{cor}\label{functional eq}
Let $E/K$ be an elliptic curve defined over a number field $K$ with complex multiplication by an order $\Z[\omega]$ in a quadratic imaginary field $F$. For any $\alpha=a+b\omega \in \Z[\omega]$ and any non-torsion points $\bm{P}=(P, \omega P) \in E(K)^{2}$, if $\nu$ is the finite place associated to the prime $\p$, the normalised N\'{e}ron local height function satisfies
\[
\widetilde{\lambda}([\alpha]P)=a^{2}\widetilde{\lambda}(P)+b^{2}\widetilde{\lambda}(\omega P)+ab\left(\widetilde{\lambda}([1+\omega]P)-\widetilde{\lambda}(P)-\widetilde{\lambda}(\omega P)\right)+\nu(\Psi_{(a, b)}(\bm{P})),
\]
where $\Psi_{(a, b)}$ is the elliptic net polynomial evaluated at $\bm{v}=(a, b)$, the vector notation for elements in $\Z[\omega]$.
\end{cor}

\textcolor{red}{To prove the explicit formula, recall that the formula from \cite{Cheon} requires the point $[n]P$ to have bad reduction, therefore it does not have a formula for $g_{r, \nu}(P)$ and requires one to first compute $\mu\coloneqq \min\left(2\nu(\psi_{r}), \nu(\phi_{r})\right)$. We have a similar situation here as well: recall that elliptic nets of rank $\geq 2$ requires minimum $3$ different basis vectors to define all terms: $\bm{e_{i}, e_{j}}$ and $\bm{e_{i+j}}$. Therefore, we also do not have control in the valuations of the net polynomials at $[r]P, [r\omega]P$ and $[r+r\omega]P$, which is why we need the following proposition. }

\begin{prop}\label{ann}
Let \( E/K \) be an elliptic curve with complex multiplication by \( \mathbb{Z}[\omega] \). Let \( \bm{P} = (P, \omega P) \), where \( P \in E(K) \setminus E_0(K) \) is a non-torsion point. Let $r=r_1+r_2\omega \in \textcolor{red}{\A(\p, P)}$ and \( \alpha \in \mathbb{Z}[\omega] \), so that \( \alpha r \in  \textcolor{red}{\A(\p, P)} \). Write \( \alpha r = X + Y\omega \), where \( X, Y \in \mathbb{Z} \). Then,
\begin{multline*}
2X^{2}\widetilde{\lambda}(P) + 2Y^{2}\widetilde{\lambda}(\omega P) 
- 2XY\big(\widetilde{\lambda}(P) + \widetilde{\lambda}(\omega P)- \widetilde{\lambda}([1+\omega]P)\big)  \\
= -\min\left(2\nu\big(\Psi_{(X, Y)}(\bm{P})\big), 
\nu\big(\Phi_{(X, Y)}(\bm{P})\big)\right).
\end{multline*}
\end{prop}

\begin{proof}
First notice that $[\alpha r]P$ is non-singular $\Mod \p$, so we always have
\[
-\min\left(2\nu\left(\Psi_{(X, Y)}(\bm{P})\right), \nu\left(\Phi_{(X, Y)}(\bm{P})\right)\right) =
\begin{cases}
    -2v(\Psi_{(X, Y)}(\bm{P})) & \text{if } [\alpha r]P \in E_0(K) \setminus E_1(K); \\
    -v(\Phi_{(X, Y)}(\bm{P})) & \text{if } [\alpha r]P \in E_1(K).
\end{cases}
\]
By Proposition \ref{functional eq}, we have 
\[
\widetilde{\lambda}([\alpha r]P) = X^{2}\widetilde{\lambda}(P) +Y^{2}\widetilde{\lambda}(\omega P) + XY\left(\widetilde{\lambda}([1+\omega]P) - \widetilde{\lambda}(P) - \widetilde{\lambda}(\omega P)\right) + \nu(\Psi_{(X, Y)}(\bm{P})).
\]

Recall that for all $P \in E_{0}(K)$, we have $\widetilde{\lambda}(P)=\frac{1}{2}\max\left(\nu\left(x(P)^{-1}\right), 0\right)$ and we always have $\nu\left(x([\alpha r]P)\right)=\nu\left(\Phi_{(X, Y)}(\bm{P})\right)-2\nu\left(\Psi_{(X, Y)}(\bm{P})\right)$.

If $[\alpha r]P \in E_{1}(K)$, then 
\[
\widetilde{\lambda}([\alpha r]P)=-\frac{1}{2}\nu\left(x([\alpha r]P)\right), \,\,\text{and}
\]
\[
2X\widetilde{\lambda}(P)+2Y\widetilde{\lambda}(\omega P)+2XY\left(\widetilde{\lambda}([1+\omega]P)-\widetilde{\lambda}(P)-\widetilde{\lambda}(\omega P)\right)=-\nu\left(\Phi_{(X, Y)}(\bm{P})\right).
\]
For \([\alpha r]P \in E_{0}(K) \setminus E_{1}(K)\), note that we must have \(\widetilde{\lambda}([\alpha r]P) = 0\). Therefore,
\[
2X\widetilde{\lambda}(P)+2Y\widetilde{\lambda}(\omega P)+2XY\left(\widetilde{\lambda}([1+\omega]P)-\widetilde{\lambda}(P)-\widetilde{\lambda}(\omega P)\right)= -2\nu\left(\Psi_{(X, Y)}(\bm{P})\right).
\]
\qedhere
\end{proof}
\textcolor{red}{Regarding the valuations of the net polynomials at $[r]P, [r\omega]P$ and $[r+r\omega]P$, we now }set the following notations:
\begin{align*}
\mu(r) & \coloneqq \min(2\nu(\psi_{r}(P)), \nu(\phi_{r}(P)))=-2\left(\widetilde{\lambda}\left([r]P\right)-\nu\left(\psi_{r}\right)\right), \\
\mu_{\omega}(r) & \coloneqq \min(2\nu(\psi_{r\omega}(P)), \nu(\phi_{r\omega}(P)))=-2\left(\widetilde{\lambda}\left([r\omega]P\right)-\nu\left(\psi_{r\omega}\right)\right), \\
\mu_{1+\omega}(r) & \coloneqq \min(2\nu(\psi_{r(1+\omega)}(P)), \nu(\phi_{r(1+\omega)}(P)))=-2\left(\widetilde{\lambda}\left([r(1+\omega)]P\right)-\nu\left(\psi_{r(1+\omega)}\right)\right),
\end{align*}
where $P$ is a singular point $\Mod \p$ in \(E(K)\) and the alternative expression using the normalised local height function follows from \Cref{functional eq} and \Cref{ann} together.  
\begin{remark}
As in \Cref{r dependency}, we would like to remind readers that the values of \(\mu(r)\), \(\mu_{\omega}(r)\), and \(\mu_{1+\omega}(r)\) depend on the choice of $r \in \textcolor{red}{\A(\p, P)}$. Even when $\textcolor{red}{\A(\p, P)}$ is principal, they depend on the choice of the generator: the Gaussian integers \(\Z[i]\) and the Eisenstein integers \(\Z[\rho]\) have units $\{\pm1, \pm i\}$ and $\{1, \rho, \rho^2\}$ respectively, where \(\rho\) is a primitive cube root of unity. In the latter case, we have \(\mu(r) = \mu(\rho r)=\mu(\rho^2 r)\) (similar for $\mu_{\omega}(r) = \mu_{1+\omega}(r)$) because \(\rho\) satisfies \(\rho^2 = -\rho - 1\). While for \(\Z[i]\), we have \(\mu(\pm i r) = \mu_{\omega}(\mp r)\), but $\mu_{1+\omega}(ir)=g_{-1+i, \nu}(P)$ instead.
\end{remark}

\subsection{The finiteness of local height function values}
In \cite{Panda}, part of the calculation relies on the fact that \( \tilde{\lambda} \left([n]P\right) \) takes values in a finite set when \( n \) is not a multiple of the order of \( P \) in the finite group \( E(L)/E_0(L) \). This leads to some cancellation of terms and results in a nice formula. Now we would like to show that this result still holds in our case.
Now with \Cref{reduction lemma}, we can imitate \cite{Panda} and show that even when the elliptic curve has CM by $\Z[\omega]$, $\lambda([\alpha]P)$ (hence $\widetilde{\lambda}([\alpha]P)$) only takes a finite set of values when $\alpha \not\in \textcolor{red}{\A(\p, P)}$. 

\begin{prop}\cite[Lemma 25]{Panda}\label{finite values}
Let $E/K$ be an elliptic curve with complex multiplication by $\Z[\omega]$. For a non-torsion point $P$ in $E(K)\setminus E_{0}(K)$, \textcolor{red}{let $\A(\p, P)$ be the annihilator of $P$ in the $\Z[\omega]$-module $E(K)/E_{0}(K)$ (as defined in \Cref{ANN}).} Let $\nu$ be a finite place in $K$ such that $E$ has additive reduction and $P$ is singular $\Mod \nu$. \textcolor{red}{Assume $r \in \A(\p, P)$}. Then for any $\alpha=a+b\omega \notin\textcolor{red}{\A(\p, P)}$, we have 
\[
\lambda\left([\alpha]P\right)=\lambda\left([r\pm\alpha]P\right).
\]
\end{prop}

\begin{proof}
By the Semi-stable reduction theorem (see \cite[Chapter VII.5, Proposition 5.4, page 197]{Silverman1}), there exists a field extension $L/K$ such that $E$ has good reduction and a finite place $w$ in $L$ that extends $\nu$. From (\ref{chg of coord}), we know the new $x$-coordinate is $x_{L}=\frac{x_{K}-\gamma}{u^2}$. Now the point $[r]P$ is non-singular $\Mod \nu$, so by \Cref{reduction lemma} we have $x_{K}([r]P)\not\equiv \gamma \Mod w$ and $w\left(x_{L}([r]P)\right)=-2w(u)<0$. However, since $[\alpha]P$ is singular $\Mod \nu$, the lemma tells us that $w\left(x_{L}([\alpha]P)\right)>-2w(u)=w\left(x_{L}([r]P)\right)$. We know there is an $\Z[\omega]$-module filtration as in (\ref{w filtration}), which this tells us that if $[r]P \in E_{i}(L_{w})$ and $[\alpha]P \in E_{j}(L_{w})$, then we have $i>j$. Therefore, $[\alpha+r]P \in E_{j}(L_{w})$. Recall that the points $[r]P$ and $[r\pm \alpha]P$ have good reduction in $E(L)$, so the formula in \Cref{formula} applies and tell us that 
\[
\lambda_{w}\left([r\pm \alpha]P\right)=\lambda_{w}\left([\pm\alpha]P\right)=\lambda_{w}\left([\alpha]P\right).
\]
Since the local height function is invariant under field extension, we also get 
\[
\lambda_{\nu}\left([r\pm \alpha]P\right)=\lambda_{\nu}\left([\alpha]P\right).
\]
\qedhere
\end{proof}

\section{Proof of \Cref{CM gcd}}
\subsection{Proof of \Cref{CM gcd}(1): points with good reduction}
For points with good reduction, we can use the result from \cite[Proposition 1.7]{Akbary}. The difference here is that if the order $\Z[\omega]$ is not a principal ideal domain, then, writing $I_K$ as the group of non-zero fractional ideals of $K$ and using the notations in \Cref{net poly}\ref{notation}, we define the quadratic form $F(\bm{P}) \colon \Z^{r} \rightarrow I_K$ by
\begin{equation}\label{F}
F_{\bm{v}}(\bm{P})=\prod_{1\leq i \leq j \leq r}A_{ij}^{v_{i}v_{j}},
\end{equation}
where
\begin{equation}
A_{ii}=D_{\bm{e_i} \cdot \bm{P}}=D_{P_{i}}, \quad \text{and} \quad 
A_{ij}=(D_{P_{i}+P_{j}})(D_{P_{i}})^{-1}(D_{P_{j}})^{-1} \quad \text{for} \quad i \neq j,
\end{equation}
and $D_{P}$ is defined by $x(P)\OO_{K}=A_{P}D_{P}^{-2}$ as in \Cref{expression}. Then for $\bm{v} \in \Z^{r}$, we can define the ideal
\[
\widehat{\Psi}_{\bm{v}}(\bm{P})=F_{\bm{v}}(\bm{P})\Psi_{\bm{v}}(\bm{P})\OO_{K}
\]
that satisfies
\begin{align*}
\widehat{\Psi}_{\bm{e_i}}(\bm{P})&=D_{\bm{e_i}\cdot \bm{P}}, \\
\widehat{\Psi}_{\bm{e_i+e_j}}(\bm{P})&=D_{(\bm{e_i}+\bm{e_j})\cdot \bm{P}}.
\end{align*}
If $K$ has class number 1, then the ideals $A_{ij}$ are principal with $A_{ij}=(\alpha_{ij})$. In this case, we can define $F(\bm{P})$ with the $\alpha$'s, and then we get an elliptic net that is scale equivalent to $\Psi_{\bm{v}}(\bm{P})$ as in \cite[Proposition 1.7]{Akbary}. If not, we can still get a similar result, except we are now comparing the exponents of prime ideals in the ideals' unique factorisations rather than valuations of elements. 

\begin{prop}\label{strong}
Let $K$ be a number field and $E/K$ be an elliptic curve defined in Weierstrass equation with coefficients in $\OO_{K}$. Let $\bm{P}=(P_{1}, \dotsc, P_{r}) \in E(K)^{n}$ be an $n$-tuple consisting of $n$ linearly independent points in $E(K)$. Let $\p$ be a prime of $K$ such that $P_{i} \Mod \p$ is non singular for $1 \leq i \leq n$. Then for all $\bm{v} \in \Z^{r}$, 
\[
d_{\p}(D_{\bm{v} \cdot \bm{P}})=d_{\p}\left(\widehat{\Psi}_{\bm{v}}(\bm{P})\right),
\]
where $d_{\p}$ is as defined in \Cref{Valuations}. In particular, if for all primes $\p$ and all $1 \leq i \leq r$ we have that $P_{i} \Mod \p$ is non-singular, then 
\[
D_{\bm{v} \cdot \bm{P}}=\widehat{\Psi}_{\bm{v}}(\bm{P}).
\]
If $K$ has class number $1$, then $\widehat{\Psi}_{\bm{v}}\left(\bm{P}\right)$ is an elliptic net that is scale equivalent to $\Psi_{\bm{v}}\left(\bm{P}\right)$.
\end{prop}

\begin{proof}
The proof mostly follows from \cite[Proposition 1.7]{Akbary}. One can verify, or refer back to \cite[Lemma 2.11]{Akbary} that the function $\varepsilon: \Z^{r} \rightarrow \Z$, defined by
\begin{equation} \label{epsilon}
\varepsilon(\bm{v})= \begin{cases}
\lambda(\bm{v} \cdot \bm{P})-\frac{1}{12}\nu(\Delta)-d_{\p}(\Psi_{\bm{v}}(\bm{P})\OO_{K}), & \text{if $\bm{v}\neq 0$,} \\
0, & \text{otherwise}.
\end{cases}
\end{equation}
is a quadratic form, where $\lambda$ is the N\'{e}ron height function associated to the prime $\p$. 
Since $\nu$ is normalised, we know $\nu(\alpha)=d_{\p}(\alpha \OO_{K})$ for $\alpha \in K$. Recall from (\ref{good reduction local height}) that if $P \Mod \p$ is non-singular, then we have 
\[
\lambda(P)=\max\left(-\frac{1}{2}d_{\p}(x(P)\OO_{K}), 0\right)+\frac{1}{12}\nu(\Delta).
\]
Observe that in this case we have 
$d_{\p}(D_{P}) = \max\left(-\tfrac{1}{2} d_{\p}(x(P)\OO_K), 0\right)$,
so for $\bm{v} \in \Z^{r} \setminus \{\bm{0}\}$,
we can rewrite the quadratic form $\varepsilon(\bm{v})$ as

\begin{align*}
\varepsilon(\bm{v})
&=\lambda(\bm{v} \cdot \bm{P})-\frac{1}{12}\nu(\Delta)-d_{\p}(\Psi_{\bm{v}}(\bm{P})\OO_{K}) \\
&=d_{\p}(D_{\bm{v} \cdot \bm{P}})-d_{\p}\left(\Psi_{\bm{v}}(\bm{P})\OO_{K}\right).
\end{align*}

Now we further define another quadratic form
\begin{align*}
\widehat{\varepsilon}: \Z^{r} \rightarrow \Z, \quad \widehat{\varepsilon}(\bm{v}) 
&= \varepsilon(\bm{v}) - d_{\p}(F_{\bm{v}}(\bm{P})) \\
&= d_{\p}(D_{\bm{v} \cdot \bm{P}}) - d_{\p}(\widehat{\Psi}_{\bm{v}}(\bm{P})),
\end{align*} 

where we remind readers that $F_{\bm{v}}(\bm{P})$ is as defined in (\ref{F}). 

The difference between two quadratic forms is also a quadratic form, hence $\hat{\varepsilon}$ is a quadratic form.  Note that for any $1 \leq i < j \leq n$, $\widehat{\varepsilon}$ satisfies
\begin{align*}
\widehat{\varepsilon}(\bm{e_{i}}) 
&= d_{\p}(D_{P_{i}}) - d_{\p}(\widehat{\Psi}_{\bm{e_{i}}}(\bm{P}))=0 \\
\widehat{\varepsilon}(\bm{e_{i}+e_{j}}) 
&= d_{\p}(D_{P_{i}+P_{j}}) - d_{\p}(\widehat{\Psi}_{\bm{e_{i}+e_{j}}}(\bm{P}))=0.
\end{align*}

Note that if a quadratic form $G \colon \Z^r \rightarrow I_{K}$ satisfies $G(\bm{v})=0$ for all $\bm{v}=\bm{e_{i}}$ and $\bm{v}=\bm{e_{i}+e_{j}}$, then G($\bm{v})=0$ for all $\bm{v}$ (see \cite[Lemma 4.5]{Stange}), so $\widehat{\varepsilon}(\bm{v})=0$ for all $\bm{v} \in \Z^{n}$, which we conclude using the definition of $\widehat{\varepsilon}$ that for all $\bm{v} \in \Z^{r}$,
\[
 d_{\p}(D_{\bm{v} \cdot \bm{P}})=d_{\p}(\widehat{\Psi}_{\bm{v}}(\bm{P})).
\]
\qedhere
\end{proof}

\begin{remark}
We would like to point out that in \cite[Theorem 1.7]{Akbary}, this theorem was only stated for $K=\Q$. However, the main ingredient of the proof is to use the local N\'{e}ron height function to construct quadratic forms in a way that measures the difference of the valuations of elliptic denominator net and net polynomial, which is not specific to any number field. Therefore, this theorem is true for general number fields $K$. In particular, it is true for imaginary quadratic fields. 
\end{remark}
\subsection{Proof of \Cref{CM gcd}(2): multiples of $P$ with good reduction}
\textcolor{red}{In this section, we assume $P \in E(K)$ has singular reduction $\Mod \p$.} The main tool for this case is the `change of basis' formula for elliptic nets by Stange in \cite[Section 5]{Stange3} (proof can be found in \cite[Proposition 4.3]{Stange}\textcolor{red}{)}, which is the following: let \( T \) be any \( n \times m \) matrix. Let \( \mathbf{P}=(P_1, P_2, \dotsc, P_{m}) \in E(K)^m \), \( \mathbf{v}=(v_1, v_2, \dotsc, v_{n}) \in \mathbb{Z}^n \). For an elliptic net $W_{E, \bm{P}} \colon \Z^m \rightarrow K$ associated to the curve $E/L$ and the points $P_1, P_2, \dotsc, P_{m}$, it satisfies
\[
W_{E, \bm{P}}(T^{\text{tr}}(\bm{v})) = W_{E, T(\bm{P})}(\bm{v}) \prod_{i=1}^n W_{E, \bm{P}}(T^{\text{tr}}(e_i))^{v_i^2 - v_i \left( \sum_{j \neq i} v_j \right)} \prod_{1 \leq i < j \leq n} W_{E, \bm{P}}(T^{\text{tr}}(e_i + e_j))^{v_i v_j},
\]
where $T^{\text{tr}}$ denotes the transpose of $T$. Now let \( \alpha = a + b\omega \), \( \beta = c + d\omega \in \mathbb{Z}[\omega] \). From \Cref{order structure}, we know that \( \omega^2 = A\omega - D \), where \( A = \{0, f\} \), with \( f \) being the conductor of \( \Z[\omega] \) and \( D \) a positive integer.
We write \( \bm{\alpha} \) and \( \bm{\beta} \) for their vector notations. We have
\[
\alpha \beta = (ac - bdD) + (ad + bc + bdA)\omega.
\]
If we take \( \bm{v} = \bm{\beta} \) in the formula above, the transpose of the matrix is

\[
T^{\text{tr}}=\begin{pmatrix}
  a & -bD\\ 
  b & a+bA
\end{pmatrix}
\]
and one can check that indeed, 
\[
T^{\text{tr}}\bm{\beta}=\begin{pmatrix}
  ac-bdD \\ 
bc+ad+bdA
	\end{pmatrix},
\]
which is the vector notation of $\alpha\beta$. On the other hand, for $\bm{P}=\begin{pmatrix}
P \\ 
\omega P
	\end{pmatrix} \in E(K)^2$, we have 
\[
T\bm{P}=\begin{pmatrix}
 aP+b\omega P \\ 
-bdP+(a+bA)\omega P
	\end{pmatrix}=
\begin{pmatrix}
 [\alpha]P \\ 
[\alpha\omega] P
	\end{pmatrix}.
\]
So overall, using net polynomial notations, we have 
\[
\Psi_{T^{\text{Tr}}\bm{\beta}}(\bm{P})=\Psi_{\beta}(T\bm{P})\Psi_{T^{\text{Tr}}(1, 0)^{\text{Tr}}}(\bm{P})^{c^2-cd}\Psi_{T^{\text{Tr}}(0, 1)^{\text{Tr}}}(\bm{P})^{d^2-cd}\Psi_{T^{\text{Tr}}(1, 1)^{\text{Tr}}}(\bm{P})^{cd}
\]
Using our notation instead, then this is equivalent to saying our `CM divsion polynomial' satisfies
\begin{equation}\label{EDS formula}
\psi_{\alpha \beta}(P)=\psi_{\beta}(\alpha P)\psi_{\alpha}(P)^{c^{2}-cd}\psi_{\alpha\omega}(P)^{d^{2}-cd}\psi_{\alpha(1+\omega)}(P)^{cd}.
\end{equation}
This means in this case, the matrix $T^{\text{tr}}$ represents the `multiplication by $[\alpha]$ map'. Note that $\End(E)$ is a commutative ring, so we can always interchange $\alpha$ and $\beta$ in (\ref{EDS formula}), although this would mean the matrix $T$ is now different. Nevertheless, now this allows us to easily compute the valuation of $\psi_{z}$ in terms of $\mu, \mu_{\omega}$ and $\mu_{1+\omega}$. 

\begin{prop}\label{good}
Let \textcolor{red}{$r$ be an element in $\A(\p, P)$}, where $r=r_1+r_2\omega$. Let $\alpha = a+ b\omega$ with $\alpha \in \Z[\omega]$, then
\[
\Min\left(2\nu\left(\psi_{\alpha r}(P)\right), \nu\left(\phi_{\alpha r}(P)\right)\right)=(a^{2}-ab)\mu(r)+(b^{2}-ab)\mu_{\omega}(r)+(ab)\mu_{1+\omega}(r).
\]
\end{prop}
\begin{proof} 
The main idea of the proof is as follows: \textcolor{red}{we first use the local height formula from \Cref{functional eq} to express $\nu\left(\psi_{\alpha}(rP)\right)$ in terms of $\widetilde{\lambda}\left([\alpha]rP\right)$, $\widetilde{\lambda}\left(rP\right)$, $\widetilde{\lambda}\left([\omega]rP\right)$ and $\widetilde{\lambda}\left([1+\omega]rP\right)$. Since $rP$ is a non-singular point $\bmod \nu$, we can use the local height formula for non-singular point in (\ref{good reduction local height}), which says
\[
\widetilde{\lambda}(rP) = \frac{1}{2} \max \left(\nu(x(rP)^{-1}), 0 \right), \text{ for all } rP \in E_0(K),
\]
to replace all the local height terms above so we can express $\nu\left(\psi_{\alpha}(rP)\right)$ in terms of the valuations of the division polynomials $\phi_{z}$ and $\psi_{z}$ instead. Finally, recall from  (\ref{EDS formula}) that we have 
\begin{equation}\label{EDS formula2}
\psi_{\alpha r}(P)=\psi_{\alpha}(rP)\psi_{r}(P)^{a^{2}-ab}\psi_{r\omega}(P)^{b^{2}-ab}\psi_{r(1+\omega)}(P)^{ab}.
\end{equation}
By substituting the expression we obtained for $\nu\left(\psi_{\alpha}(rP)\right)$ into (\ref{EDS formula2}) and carefully handle the cases of whether the points reduce to the identity modulo $\p$ or not, we yield the result.} Overall, we have $5$ different cases to examine:\\

\noindent\textbf{Case I:} $rP \equiv O \bmod \p$, so $r\omega P \equiv O \bmod \p$ as well. \\
\textcolor{red}{First, \Cref{functional eq} tells us that}
\begin{flalign*}
\nu\left(\psi_{\alpha}(rP)\right)
&= \widetilde{\lambda}\left([\alpha]rP\right) - a^{2}\widetilde{\lambda}(rP) - b^{2}\widetilde{\lambda}([\omega]rP) - ab\left(\widetilde{\lambda}([1+\omega]rP) - \widetilde{\lambda}(rP) - \widetilde{\lambda}([\omega]rP)\right) &&\\
&= \nu\left(\psi_{\alpha r}\right) - \frac{1}{2}\nu\left(\phi_{\alpha r}\right)
- a^{2}\left(\nu\left(\psi_{r}\right) - \frac{1}{2}\nu\left(\phi_{r}\right)\right)
- b^{2}\left(\nu\left(\psi_{\omega r}\right) - \frac{1}{2}\nu\left(\phi_{\omega r}\right)\right) &&\\
&\phantom{=} - ab\left(\nu\left(\psi_{(1+\omega)r}\right) - \frac{1}{2}\nu\left(\phi_{(1+\omega)r}\right)
- \left(\nu\left(\psi_{r}\right) - \frac{1}{2}\nu\left(\phi_{r}\right)\right)
- \left(\nu\left(\psi_{\omega r}\right) - \frac{1}{2}\nu\left(\phi_{\omega r}\right)\right)
\right) &&\\
&= \nu\left(\psi_{\alpha r}\right) - \frac{1}{2}\nu\left(\phi_{\alpha r}\right)
- (a^{2} - ab)\left(\nu\left(\psi_{r}\right) - \frac{1}{2}\nu\left(\phi_{r}\right)\right) &&\\
&\phantom{=} - (b^{2} - ab)\left(\nu\left(\psi_{\omega r}\right) - \frac{1}{2}\nu\left(\phi_{\omega r}\right)\right)
- ab\left(\nu\left(\psi_{(1+\omega)r}\right) - \frac{1}{2}\nu\left(\phi_{(1+\omega)r}\right)\right). &&
\end{flalign*}

Now substitute the expression of $\nu\left(\psi_{\alpha}(rP)\right)$ into \textcolor{red}{(\ref{EDS formula2})}, we have 
\[
\nu\left(\psi_{\alpha r}\right)=\nu\left(\psi_{\alpha r}\right)-\frac{1}{2}\nu\left(\phi_{\alpha r}\right)+\frac{a^{2}-ab}{2}\nu\left(\phi_{r}\right)+\frac{b^{2}-ab}{2}\nu\left(\phi_{r\omega}\right)+\frac{ab}{2}\nu\left(\textcolor{red}{\phi}_{r+r\omega}\right).
\]
Rearrange and multiply both sides by $2$, then we have 
\[
\nu\left(\phi_{\alpha r}\right)=(a^{2}-ab)\nu\left(\phi_{r}\right)+(b^{2}-ab)\nu\left(\phi_{r\omega}\right)+(ab)\nu\left(\textcolor{red}{\phi}_{r+r\omega}\right).
\]
Since $rP \equiv r\omega P \equiv r(1+\omega)P \equiv O \bmod \p$, we also have $\alpha rP\equiv O \bmod \p$ \textcolor{red}{and therefore the following holds:}
\[
2\nu(\psi_{r})\geq \nu(\phi_{r}), 2\nu(\psi_{r\omega})\geq \nu(\phi_{r\omega}),  2\nu(\psi_{r+r\omega})\geq \nu(\phi_{r+r\omega})\,\, \text{and} \,\,2\nu(\psi_{\alpha r})\geq \nu(\phi_{\alpha r}) .
\]
Overall, when $rP \equiv O \mod \p$,
\begin{flalign*}
&\Min\left(2\nu\left(\psi_{\alpha r}(P)\right), \nu\left(\phi_{\alpha r}(P)\right)\right) &&\\
&=\nu(\phi_{\alpha r}) &&\\
&=(a^{2}-ab)\mu(r)+(b^{2}-ab)\mu_{r\omega}(r)+(ab)\mu_{r+r\omega}(r).&&
\end{flalign*}

\noindent\textbf{Case II:} $rP, r\omega P, r(1+\omega)P$ and $\alpha rP \not\equiv O \bmod \p$ \\
We start by considering the term $\nu(\psi_{\alpha}(rP))$. Recall that if a point $P \in E_{0}(K)$ satisfies $P \not\equiv O \mod \p$, then $\widetilde{\lambda}(P) = 0$. Hence, \textcolor{red}{using} (\ref{good reduction local height}) and noting that we assumed $\alpha rP \not\equiv O \mod \p$, we have
\begin{flalign*}
\nu(\psi_{\alpha}(rP))
&= \widetilde{\lambda}([\alpha r]P) - a^{2}(0) - b^{2}(0) - ab\left(0-0-0\right) &&\\
&=\widetilde{\lambda}([\alpha r]P) &&\\
&=\frac{1}{2}\max\left(\nu\left(x(\alpha rP)^{-1}\right), 0\right) &&\\
&=0.&& 
\end{flalign*}
Then (\ref{EDS formula2}) tells us that 
\begin{flalign*}
&\indent\Min\left(2\nu\left(\psi_{\alpha r}(P)\right), \nu\left(\phi_{\alpha r}(P)\right)\right) &&\\
&=2\nu\left(\psi_{\alpha r}\right) &&\\
&=0+2(a^{2}-ab)\nu\left(\psi_{r}\right)+2(b^2-ab)\nu\left(\psi_{r\omega}\right)+(2ab)\nu\left(\psi_{r(1+\omega)}\right) &&\\
&=(a^{2}-ab)\mu(r)+(b^2-ab)\mu_{\omega}(r)+(ab)\mu_{1+\omega}(r),&&
\end{flalign*}
\textcolor{red}{where the last equality follows from the fact that 
\[
2\nu(\psi_{r})\leq \nu(\phi_{r}), 2\nu(\psi_{r\omega})\leq \nu(\phi_{r\omega}) \,\, \text{and}\,\,  2\nu(\psi_{r+r\omega})\leq \nu(\phi_{r+r\omega})
\]
under the assumption that none of the points reduces to the identity modulo $\p$. 
}

\noindent\textbf{Case III:} $rP, r\omega P, r(1+\omega)P \not\equiv O \bmod \p$ but $\alpha rP \equiv O \bmod \p$ \\
\textcolor{red}{As in Case II}, (\ref{good reduction local height}) tells us that 
\[
\nu(\psi_{\alpha}(rP))=\widetilde{\lambda}([\alpha r]P)=\frac{1}{2}\max\left(\nu\left(x(\alpha rP)^{-1}\right), 0\right)=\nu(\psi_{\alpha r})-\frac{1}{2}\nu(\phi_{\alpha r}). 
\]
Substitute this into (\ref{EDS formula2}) to give
\[
\nu\left(\psi_{\alpha r}\right)=\nu\left(\psi_{\alpha r}\right)-\frac{1}{2}\nu\left(\phi_{\alpha r}\right)\textcolor{red}{+(a^{2}-ab)\nu\left(\psi_{r}\right)+(b^{2}-ab)\nu\left(\psi_{r\omega}\right)+ab\nu\left(\psi_{r+r\omega}\right)}.
\]
Rearrange and multiply both sides by 2, we then get 
\[
\nu\left(\phi_{\alpha r}\right)=2(a^{2}-ab)\nu\left(\psi_{r}\right)+2(b^{2}-ab)\nu\left(\psi_{r\omega}\right)+(2ab)\nu\left(\psi_{r+r\omega}\right).
\]
This shows that 
\begin{flalign*}
&\indent\Min\left(2\nu\left(\psi_{\alpha r}(P)\right), \nu\left(\phi_{\alpha r}(P)\right)\right)&& \\
&=\nu\left(\phi_{\alpha r}\right) &&\\
&=\textcolor{red}{2(a^{2}-ab)\nu\left(\psi_{r}\right)+2(b^{2}-ab)\nu\left(\psi_{r\omega}\right)+(2ab)\nu\left(\psi_{r+r\omega}\right)} &&\\
&=(a^{2}-ab)\mu(r)+(b^2-ab)\mu_{\omega}(r)+(ab)\mu_{1+\omega}(r),&&
\end{flalign*}
\textcolor{red}{where the last equality follows from the assumption that
$2\nu(\psi_r)\leq \nu(\phi_r)$,
$2\nu(\psi_{r\omega})\leq \nu(\phi_{r\omega})$, and
$2\nu(\psi_{r+r\omega})\leq \nu(\phi_{r+r\omega})$.}

\noindent\textbf{Case IV:} $rP \not\equiv O \bmod \p$, but $r\omega P \equiv O \bmod \p$ and so $r(1+\omega)P \equiv rP \not\equiv O \bmod \p$ \\
Again, we start by considering $\nu(\psi_{\alpha}(rP))$. Using (\ref{good reduction local height}), we have 
\begin{flalign*}
\nu(\psi_{\alpha}(rP)) 
&=\widetilde{\lambda}([\alpha r]P)-0-b^{2}\widetilde{\lambda}(\omega rP)-ab\left(0-0-\widetilde{\lambda}(\omega rP)\right) &&\\
&=\widetilde{\lambda}([\alpha r]P)-(b^{2}-ab)\widetilde{\lambda}(\omega rP) &&\\
&=\widetilde{\lambda}([\alpha r]P)-(b^{2}-ab)\left(\nu(\psi_{\omega r})-\frac{1}{2}\nu(\phi_{\omega r})\right),&&
\end{flalign*}
where the last equality follows from the fact that $\textcolor{red}{r\omega P} \equiv O \bmod \p$. Now we want to substitute this back into (\ref{EDS formula2}), \textcolor{red}{which depends on the behaviour of $\alpha rP$ reduction modulo $\p$}.

\noindent\textbf{Case IV.1:} $\alpha rP \equiv O \bmod \p$ \\
Since $\alpha rP$ is a non-singular point $\bmod \p$, we have 
\[
\nu(\psi_{\alpha}(rP)) =\left(\nu(\psi_{\alpha r})-\frac{1}{2}\nu(\phi_{\alpha r})\right)-(b^{2}-ab)\left(\nu(\psi_{\omega r})-\frac{1}{2}\nu(\phi_{\omega r})\right).
\]
In this case we have 
\textcolor{red}{
\[
2\nu(\psi_{\alpha r}) \geq \nu(\phi_{\alpha r}), 
2\nu(\psi_{r})\leq \nu(\phi_{r}), 
2\nu(\psi_{r\omega})\geq \nu(\phi_{r\omega}),\text{ and }
2\nu(\psi_{r+r\omega})\leq \nu(\phi_{r+r\omega}),
\]
so if we substitute the expression for $\nu(\psi_{\alpha}(rP))$ back into (\ref{EDS formula2}) and rearrange as in Case I, we can conclude that 
}
\begin{flalign*}
&\indent\Min\left(2\nu\left(\psi_{\alpha r}(P)\right), \nu\left(\phi_{\alpha r}(P)\right)\right) &&\\
&=\nu(\phi_{\alpha r}) \\ 
&=2(a^{2}-ab)\nu(\textcolor{red}{\phi_{r}})+(b^{2}-ab)\nu(\phi_{\omega r})+2ab\nu(\phi_{r+r\omega})&&\\
&=(a^{2}-ab)\mu(r)+(b^{2}-ab)\mu_{\omega}(r)+ab\mu_{1+\omega}(r).&&
\end{flalign*}

\textbf{Case IV.2:} $\alpha rP \not\equiv O \bmod \p$ \\
Then as shown in Case III, we now have $\widetilde{\lambda}([\alpha r]P)=0$ \textcolor{red}{and the following inequalities hold:
\[
2\nu(\psi_{\alpha r}) \leq \nu(\phi_{\alpha r}), 
2\nu(\psi_{r})\leq \nu(\phi_{r}),
2\nu(\psi_{r\omega})\geq \nu(\phi_{r\omega}), { and }
2\nu(\psi_{r+r\omega})\leq \nu(\phi_{r+r\omega}).
\]
}
Therefore, substitute the expression for $\nu(\psi_{\alpha}(rP))$ back into (\ref{EDS formula2}) and conclude that 
\begin{flalign*}
&\indent\Min\left(2\nu\left(\psi_{\alpha r}(P)\right), \nu\left(\phi_{\alpha r}(P)\right)\right) &&\\
&=2\nu(\psi_{\alpha r}) &&\\
&=(a^{2}-ab)\nu(\textcolor{red}{\phi_{r}})+(b^{2}-ab)\nu(\phi_{\omega r})+2ab\nu(\phi_{r+r\omega})&&\\
&=(a^{2}-ab)\mu(r)+(b^{2}-ab)\mu_{\omega}(r)+ab\mu_{1+\omega}(r).&&
\end{flalign*}

\textbf{Case V:} $rP, r\omega P \not\equiv O \bmod \p$, but $r(1+\omega)P \equiv O \bmod \p$ \\
By the assumption, we have $\widetilde{\lambda}(rP) = \widetilde{\lambda}(r\omega P) = 0$. Together with (\ref{good reduction local height}) and (\ref{good reduction local height}), this gives 
\begin{flalign*}
\nu(\psi_{\alpha}(rP)) 
&=\widetilde{\lambda}([\alpha r]P)-ab\widetilde{\lambda}([r(1+\omega)]P) &&\\
&=\widetilde{\lambda}([\alpha r]P)-ab\left[\nu(\psi_{r+r\omega})-\frac{1}{2}\nu(\phi_{r+r\omega})\right].
\end{flalign*}
As in Case IV, the final result depends on the behaviour of $\alpha rP$ reduction modulo $\p$. \\

\textbf{Case V.1:} $\alpha rP \equiv O \bmod \p$ \\
Then we have 
\[
\nu(\psi_{\alpha}(rP)) =\left(\nu(\psi_{\alpha r})-\frac{1}{2}\nu(\phi_{\alpha r})\right)-ab\left(\nu(\psi_{r+r\omega})-\frac{1}{2}\nu(\phi_{r+r\omega})\right).
\]

Substitute this back into (\ref{EDS formula2}), rearrange and simplify to get 
\[
\nu\left(\phi_{\alpha r}\right)=2(a^{2}-ab)\nu\left(\psi_{r}\right)+2(b^{2}-ab)\nu\left(\psi_{r\omega}\right)+(ab)\nu\left(\phi_{r+r\omega}\right).
\]

Finally, recall from the assumption that we have the following inequalities:
\[
2\nu(\psi_{\alpha r}) \geq \nu(\psi_{\alpha r}), \quad
2\nu(\psi_{r})\leq \nu(\phi_{r}), \quad
2\nu(\psi_{r\omega})\leq \nu(\phi_{r\omega}), \quad \text{and} \quad
2\nu(\psi_{r+r\omega})\geq \nu(\phi_{r+r\omega}).
\]
Therefore, we conclude that
\begin{flalign*}
&\indent\Min\left(2\nu\left(\psi_{\alpha r}(P)\right), \nu\left(\phi_{\alpha r}(P)\right)\right) &&\\
&=\nu(\phi_{\alpha r}) \\ 
&=2(a^{2}-ab)\nu\left(\psi_{r}\right)+2(b^{2}-ab)\nu\left(\psi_{r\omega}\right)+(ab)\nu\left(\phi_{r+r\omega}\right)&&\\
&=(a^{2}-ab)\mu(r)+(b^{2}-ab)\mu_{\omega}(r)+ab\mu_{1+\omega}(r).&&
\end{flalign*}

\textbf{Case V.2:} $\alpha rP \not\equiv O \bmod \p$ \\
As in Case IV.2, we now have $\widetilde{\lambda}([\alpha r]P)=0$ and
\[
2\nu(\psi_{\alpha r}) \geq \nu(\phi_{\alpha r}), \quad
2\nu(\psi_{r})\leq \nu(\phi_{r}), \quad
2\nu(\psi_{r\omega})\geq \nu(\phi_{r\omega}), \quad \text{and} \quad
2\nu(\psi_{r+r\omega})\leq \nu(\phi_{r+r\omega}).
\]
Therefore, substitute the expression for $\nu(\psi_{\alpha}(rP))$ back into (\ref{EDS formula2}) and conclude that 
\begin{flalign*}
&\indent\Min\left(2\nu\left(\psi_{\alpha r}(P)\right), \nu\left(\phi_{\alpha r}(P)\right)\right) &&\\
&=2\nu(\psi_{\alpha r}) &&\\
&=2(a^{2}-ab)\nu(\psi_{r})+2(b^{2}-ab)\nu(\psi_{\omega r})+ab\nu(\phi_{r+r\omega})&&\\
&=(a^{2}-ab)\mu(r)+(b^{2}-ab)\mu_{\omega}(r)+ab\mu_{1+\omega}(r).&&
\end{flalign*}
\qedhere
\end{proof}

\subsection{Proof of \Cref{CM gcd}(2): multiples of $P$ with bad reduction}\label{Bad multiple}
\textcolor{red}{As in the previous section, we continue to assume $P \in E(K)$ has singular reduction $\Mod \p$.} In \cite{Cheon}, Cheon and Hahn utilised the recurrence relation (\ref{recurrence}) for the division polynomials \(\{\psi_{n}\}_{n \in \N}\) to derive an explicit formula \textcolor{red}{for the cancellation exponent}. They achieved this by computing the valuation of each term in (\ref{recurrence}) using the properties of division polynomials. In contrast, our approach involves repeatedly applying \Cref{functional eq} and empirically identifying which terms cancel each other out. 
\begin{lemma}\label{alphar}
Let $\alpha = a+b\omega \in \Z[\omega], r=r_1+r_2\omega \in \textcolor{red}{\A(\p, P)}$. Then 
\[
\nu\left(\psi_{\alpha r+1}\psi_{\alpha r-1}\right)=(a^{2}-ab)\mu(r)+(b^{2}-ab)\mu_{\omega}(r)+(ab)\mu_{1+\omega}(r).
\]
\end{lemma}
\begin{proof}
Recall that we have two cases to consider: for orders in quadratic imaginary fields, their generators could have two different behaviours: either $\omega^2=-D$, or $\omega^2=\omega -D$ for some square-free integer $D$. \\
\textbf{Case I:} $\omega^{2}=-D$ for some positive integer $D$.\\
First take note of the following:
\begin{itemize}
\item $\alpha r=(ar_1-br_2D)+(ar_2+br_1)\omega$,
\item$ r\omega = -r_2D+r_1\omega$,
\item $r(1+\omega)=(r_1-r_2D)+(r_1+r_2)\omega$.
\end{itemize}
Apply \Cref{functional eq} to the points $[\alpha r +1]P$ and $[\alpha r-1]P$ to get 
\begin{flalign*}
&\quad\nu\left(\psi_{\alpha r+1}\psi_{\alpha r-1}\right) &&\\
&=2\widetilde{\lambda}(P)-2\left((ar_1-br_2D)^2+1\right)\widetilde{\lambda}(P)-2(ar_2+br_1)^2\widetilde{\lambda}(\omega P) \\
&\quad-2(ar_1-br_2D)(ar_2+br_1)\left(\widetilde{\lambda}([1+\omega]P)-\widetilde{\lambda}(P)-\widetilde{\lambda}(\omega P)\right)&&\\
&=-2(a^2r_{1}^{2}-2abr_1r_2D+b^2r_2^2D^2)\widetilde{\lambda}(P)-2(a^2r_{2}^{2}+2abr_1r_2+b^2r_1^2)\widetilde{\lambda}(\omega P)&&\\
&\quad-2(a^2r_{1}r_{2}+abr_{1}^{2}-abr_{1}^{2}-abr_{2}^{2}D-b^{2}r_1r_2D)\left(\widetilde{\lambda}([1+\omega]P)-\widetilde{\lambda}(P)-\widetilde{\lambda}(\omega P)\right)&&\\
\intertext{Now rearrange the terms and we will see some familiar expressions:}
&=-2a^2r_1^2\widetilde{\lambda}(P)-2a^2r_2^2\widetilde{\lambda}(\omega P)-2ra^2_1r_2\left(\widetilde{\lambda}([1+\omega]P)-\widetilde{\lambda}(P)-\widetilde{\lambda}(\omega P)\right) &&\\
&\quad-2b^2r_2^2D^2\widetilde{\lambda}(P)-2b^2r_1^2\widetilde{\lambda}(\omega P)+2b^2r_1r_2D\left(\widetilde{\lambda}([1+\omega]P)-\widetilde{\lambda}(P)-\widetilde{\lambda}(\omega P)\right) &&\\
&\quad-2ab\left(-2r_1r_2D\widetilde{\lambda}(P)+2r_1r_2\widetilde{\lambda}(\omega P)+(r_1^2-r_2^2D)\left(\widetilde{\lambda}([1+\omega]P)-\widetilde{\lambda}(P)-\widetilde{\lambda}(\omega P)\right)\right).
\end{flalign*}

Note that applying \Cref{ann} to the points $[r]P$ and $[r \omega]P$ produces the terms in the first two lines of the equality sign above, which correspond to $\mu(P)$ and $\mu_{\omega}(P)$, respectively. Thus, we have
\begin{flalign*}
&\quad\nu\left(\psi_{\alpha r+1}\psi_{\alpha r-1}\right) &&\\
&=a^{2}\cdot\mu(r)+b^{2}\cdot \mu_{\omega}(r)&&\\
&\quad-2ab\left(-2r_1r_2D\widetilde{\lambda}(r)+2r_1r_2\widetilde{\lambda}(\omega P)+(r_1^2-r_2^2D)\left(\widetilde{\lambda}([1+\omega]P)-\widetilde{\lambda}(P)-\widetilde{\lambda}(\omega P)\right)\right).
\end{flalign*}

For the remaining terms, we once again use \Cref{functional eq} to find an expression for $\widetilde{\lambda}\left([r(1+\omega)]P\right)-\nu\left(\psi_{r(1+\omega)}\right)$:

\begin{flalign*}
&\quad \widetilde{\lambda}\left([r(1+\omega)]P\right) - \nu\left(\psi_{r(1+\omega)}\right) &&\\
&= (r_1 - r_2D)^2\widetilde{\lambda}(P) + (r_1 + r_2)^2\widetilde{\lambda}(\omega P) &&\\
&\quad + (r_1^2 + r_1r_2 - r_1r_2D - r_2^2D)\left(\widetilde{\lambda}([1+\omega]P) - \widetilde{\lambda}(P) - \widetilde{\lambda}(\omega P)\right) &&\\
&= (r_1^2 - 2r_1r_2D + r_2^2D^2)\widetilde{\lambda}(P) + (r_1^2 + 2r_1r_2 + r_2^2)\widetilde{\lambda}(\omega P) &&\\
&\quad + (r_1^2 + r_1r_2 - r_1r_2D - r_2^2D)\left(\widetilde{\lambda}([1+\omega]P) - \widetilde{\lambda}(P) - \widetilde{\lambda}(\omega P)\right). &&\\
\intertext{As before, by rearranging the terms and applying \Cref{ann} to the points \( [r]P \) and \( [r\omega]P \), we get, in order, \( -\frac{1}{2}\mu(P) \), \( \frac{1}{2}\mu_{\omega}(P) \), and the remaining terms:}
&= r_1^2\widetilde{\lambda}(P) + r_2^2\widetilde{\lambda}(\omega P) + r_1r_2\left(\widetilde{\lambda}([1+\omega]P) - \widetilde{\lambda}(P) - \widetilde{\lambda}(\omega P)\right) &&\\
&\quad + r_2^2D^2\widetilde{\lambda}(P) + r_1^2\widetilde{\lambda}(\omega P) - r_1r_2D\left(\widetilde{\lambda}([1+\omega]P) - \widetilde{\lambda}(P) - \widetilde{\lambda}(\omega P)\right) &&\\
&\quad - 2r_1r_2D\widetilde{\lambda}(P) + 2r_1r_2\widetilde{\lambda}(\omega P) + (r_1^2 - r_2^2D)\left(\widetilde{\lambda}([1+\omega]P) - \widetilde{\lambda}(P) - \widetilde{\lambda}(\omega P)\right) &&\\
&= -\frac{1}{2}\mu(P) - \frac{1}{2}\mu_{\omega}(P) &&\\
&\quad - 2r_1r_2D\widetilde{\lambda}(P) + 2r_1r_2\widetilde{\lambda}(\omega P) + (r_1^2 - r_2^2D)\left(\widetilde{\lambda}([1+\omega]P) - \widetilde{\lambda}(P) - \widetilde{\lambda}(\omega P)\right).
\end{flalign*}

\noindent Hence, an alternative expression for the remaining terms is $\frac{1}{2}\left(\mu(r) + \mu_{\omega}(r) - \mu_{1+\omega}(r)\right)$. Substituting this back into \( \nu\left(\psi_{\alpha r+1}\psi_{\alpha r-1}\right) \), we obtain
\begin{flalign*}
&\quad \nu\left(\psi_{\alpha r+1}\psi_{\alpha r-1}\right) &&\\
&= a^2\cdot\mu(P) + b^2\cdot \mu_{\omega}(P) - 2ab\left(\frac{1}{2}\left(\mu(P) + \mu_{\omega}(P) - \mu_{1+\omega}(P)\right)\right) &&\\
&= (a^2 - ab)\mu(r) + (b^2 - ab)\mu_{\omega}(r) + (ab)\mu_{1+\omega}(r).
\end{flalign*}

\textbf{Case II:} $\omega^2=f\omega -D$ for some positive integer $D$ \\
Again, we need to first take note of some expressions:
\begin{itemize}
\item $\alpha r=(ar_1-br_2D)+(ar_2+br_1+br_2f)\omega$,
\item$ r\omega = -r_2D+(r_1+r_2f)\omega$,
\item $r(1+\omega)=(r_1-r_2D)+(r_1+r_2+r_2f)\omega$.
\end{itemize}

The proof proceeds exactly as in Case I: first, use \Cref{functional eq} to expand $\nu\left(\psi_{\alpha r+1}\psi_{\alpha r-1}\right)$, then apply \Cref{ann} to identify $\mu(P)$ and $\mu_{\omega}(P)$. Finally, clear up the remaining terms by using \Cref{functional eq} on $\widetilde{\lambda}\left([r(1+\omega)]P\right)-\nu\left(\psi_{r(1+\omega)}\right)$. Since the strategy is the same, we will omit intermediate explanations for conciseness.
\begin{flalign*}
&\quad \widetilde{\lambda}\left([r(1+\omega)]P\right)-\nu\left(\psi_{r(1+\omega)}\right) &&\\
&=(r_1^2 - 2r_1r_2D + r_2^2D^2)\widetilde{\lambda}(P) + (f^2 r_2^2 + 2 f r_1 r_2 + 2 f r_2^2 + r_1^2 + 2 r_1 r_2 + r_2^2)\widetilde{\lambda}(\omega P) && \\ 
&\quad + (-D f r_2^2 - D r_1 r_2 - D r_2^2 + f r_1 r_2 + r_1^2 + r_1 r_2)\left(\widetilde{\lambda}([1+\omega]P) - \widetilde{\lambda}(P) - \widetilde{\lambda}(\omega P)\right) && \\ 
&= -\frac{1}{2}\mu(r) -\frac{1}{2}\mu_{\omega}(r)-2r_1r_2D\widetilde{\lambda}(P) + (2r_1r_2 + 2fr_2^2)\widetilde{\lambda}(\omega P) &&\\ 
&\quad  + (r_1^2 +fr_1r_2 - r_2^2D)\left(\widetilde{\lambda}([1+\omega]P) - \widetilde{\lambda}(P) - \widetilde{\lambda}(\omega P)\right).
\end{flalign*}
Therefore, we get 
\begin{flalign*}
&\quad \nu\left(\psi_{\alpha r+1}\psi_{\alpha r-1}\right) &&\\
&=-2(a^2r_{1}^{2}-2abr_1r_2D+b^2r_2^2D^2)\widetilde{\lambda}(P) &&\\
&\quad - 2(a^2 r_2^2 + 2 a b f r_2^2 + 2 a b r_1 r_2 + b^2 f^2 r_2^2 + 2 b^2 f r_1 r_2 + b^2 r_1^2)\widetilde{\lambda}(\omega P) &&\\
&\quad -2(a^2 r_1 r_2 - a b D r_2^2 + a b f r_1 r_2 + a b r_1^2 - b^2 D f r_2^2 - b^2 D r_1 r_2)\left(\widetilde{\lambda}([1+\omega]P)-\widetilde{\lambda}(P)-\widetilde{\lambda}(\omega P)\right) &&\\
&=a^2\cdot \mu(r) + b^2 \cdot \mu_{\omega}(r) &&\\
&\quad - 2ab\left(-2r_1r_2D\widetilde{\lambda}(P) + (2r_1r_2+2fr_2^2)\widetilde{\lambda}(\omega P) + (r_1^2+fr_1r_2-r_2^2D)\left(\widetilde{\lambda}([1+\omega]P)-\widetilde{\lambda}(P)-\widetilde{\lambda}(\omega P)\right)\right) &&\\
&=a^2\cdot \mu(r) + b^2 \cdot \mu_{\omega}(r) - 2ab\left(\frac{1}{2}\left(\mu(r)+\mu_{\omega}(r)-\mu_{1+\omega}(r)\right)\right) &&\\
&=(a^{2}-ab)\mu(r) + (b^{2}-ab)\mu_{\omega}(r) + (ab)\mu_{1+\omega}(r),
\end{flalign*}
as required, which completes the proof.
\qedhere
\end{proof}
In fact, this is the only term that needs to be evaluated before obtaining the main result. The main difficulty, once again, lies in determining the additional terms that must be included in the original formula by Cheon and Hahn.
\begin{prop}\label{NotAnn}
Let $\alpha = a+b\omega \in \Z[\omega], r=r_1+r_2\omega\in \textcolor{red}{\A(\p, P)}$. Further let $\beta=c+d\omega$ such that $\beta \not\in \textcolor{red}{\A(\p, P)}$. Then 
\begin{flalign*}
&\Min\left(2\nu\left(\psi_{\alpha r\pm\beta}(P)\right), \nu\left(\phi_{\alpha r\pm\beta}(P)\right)\right) &&\\
&=2\nu\left(\psi_{\alpha r \textcolor{red}{\pm} \beta}\right) &&\\
&=2 \nu\left(\psi_{\beta}\right) \pm 2a\left(\nu\left(\frac{\psi_{\beta}}{\psi_{r - \beta}}\right) + \frac{\mu}{2}\right)\pm2b\left(\nu\left(\frac{\psi_{\beta}}{\psi_{r \omega - \beta}}\right)+\frac{\mu_{\omega}}{2}\right) &&\\
& \quad+(a^{2} - ab) \mu(r) + (b^{2} - ab) \mu_{\omega}(r) + (ab)\mu_{1+\omega}(r).
\end{flalign*}
\end{prop}

\begin{proof}
Since the point $[\alpha r +\beta]P$ is singular $\Mod \p$, we know $[\alpha r +\beta]P \not\equiv O \Mod \p$ and therefore
\[
\Min\left(2\nu\left(\psi_{\alpha r+\beta}(P)\right), \nu\left(\phi_{\alpha r+\beta}(P)\right)\right)=2\nu\left(\psi_{\alpha r+\beta}(P)\right).
\]
As in \Cref{alphar}, we have two cases, although the strategies are identically the same.  \\
\textbf{Case I:} $\omega^{2}=-D$ for some positive integer $D$.\\
For convenience, take note of the following expansions:
\begin{itemize}
\begin{multicols}{2}
\item $\alpha r=(ar_1-br_2D)+(ar_2+br_1)\omega $
\item $r-\beta=(r_1-c)+(r_2-d)\omega$
\item $(r_1-c)(r_2-d)=r_1r_2-dr_1-cr_2+d$
\columnbreak
\item $r\omega=-r_2D+r_1\omega$
\item $r\omega - \beta =(-r_2D-c)+(r_1-d)\omega$
\item $(-r_2D-c)(r_1-d)=-r_1r_2D+r_2dD-cr_1+cd$ 
\end{multicols}
\end{itemize}
We first use \Cref{functional eq} to evaluate $a\left(\nu\left(\frac{\psi_{\beta}}{\psi_{r - \beta}}\right)\right)$ and $b\left(\nu\left(\frac{\psi_{\beta}}{\psi_{r\omega - \beta}}\right)\right)$.
\begin{flalign*}
&\quad a\left(\nu\left(\psi_{\beta}\right)-\nu\left(\psi_{r - \beta}\right)\right) &&\\
&= a\left( \tilde{\lambda}\left([\beta]P\right) - c^2\tilde{\lambda}(P) - d^2\tilde{\lambda}(\omega P) - cd\left[\widetilde{\lambda}([1+\omega]P) - \widetilde{\lambda}(P) - \widetilde{\lambda}(\omega P)\right] \right) &&\\
&\quad - a\left( \widetilde{\lambda}\left([r-\beta]P\right) - (r_1^2 + 2r_1c + c^2)\widetilde{\lambda}\left(P\right) - (r_2^2 - 2r_2d + d^2)\tilde{\lambda}\left(\omega P\right) \right) &&\\
&\quad - a\left((-r_1r_2 + dr_1 - cr_2 + cd)\left(\widetilde{\lambda}([1+\omega]P) - \widetilde{\lambda}(P) - \widetilde{\lambda}(\omega P)\right) \right). 
\end{flalign*}

Note that terms in $c$ and $d$ only cancel out each other, while $\widetilde{\lambda}\left([\beta]P\right)$ cancels with $\widetilde{\lambda}\left([r\omega-\beta]P\right)$ by \Cref{finite values}. By noticing the terms in $r_1$ and $r_2$ resemble $-\frac{1}{2}\mu(r)$, we obtain

\begin{flalign*}
&\quad a\left(\nu\left(\psi_{\beta}\right)-\nu\left(\psi_{r - \beta}\right)\right) &&\\
&= a\left(-\frac{\mu(r)}{2}\right) + 2car_1\widetilde{\lambda}(P) - 2dar_2\widetilde{\lambda}(\omega P) - (dar_1 + acr_2) \left( \widetilde{\lambda}([1 + \omega]P) - \widetilde{\lambda}(P) - \widetilde{\lambda}(\omega P) \right).
\end{flalign*}
Next, we evaluate $b\left[\nu\left(\psi_{\beta}\right)-\nu\left(\psi_{r\omega - \beta}\right)\right] $.
\begin{flalign*}
&\quad b\left(\nu\left(\psi_{\beta}\right)-\nu\left(\psi_{r\omega - \beta}\right)\right) &&\\
&= b\left( \widetilde{\lambda}\left([\beta]P\right) - c^2\widetilde{\lambda}(P) - d^2\widetilde{\lambda}(\omega P) - cd\left[\widetilde{\lambda}([1+\omega]P) - \widetilde{\lambda}(P) - \widetilde{\lambda}(\omega P)\right] \right) &&\\
&\quad - b\left(\widetilde{\lambda}\left([r\omega-\beta]P\right) - (r_2^2D^2 + 2r_2cD + c^2)\widetilde{\lambda}\left(P\right) - (r_1^2 - 2r_1d + d^2)\widetilde{\lambda}\left(\omega P\right) \right) &&\\
&\quad - b\left(-( -r_1r_2D + r_2dD - cr_1 + cd)\left[\widetilde{\lambda}([1+\omega]P) - \widetilde{\lambda}(P) - \widetilde{\lambda}(\omega P)\right] \right).
\end{flalign*}

\noindent Similarly, the terms in $c$ and $d$ cancel out each other, while $\widetilde{\lambda}\left([\beta]P\right)$ cancels with $\widetilde{\lambda}\left([r\omega-\beta]P\right)$ by \Cref{finite values}. By noticing the terms in $r_1$ and $r_2D$ resemble $-\frac{1}{2}\mu_{\omega}(r)$, we get

\begin{flalign*}
&\quad b\left(\nu\left(\psi_{\beta}\right)-\nu\left(\psi_{r\omega - \beta}\right)\right) &&\\
&= b\left(-\frac{\mu_{\omega}(r)}{2}\right) + 2cbr_2D\widetilde{\lambda}(P) - 2dbr_1\widetilde{\lambda}(\omega P) + (dbr_2D - bcr_1) \left( \widetilde{\lambda}([1 + \omega]P) - \widetilde{\lambda}(P) - \widetilde{\lambda}(\omega P) \right).
\end{flalign*}
Now, if we use \Cref{functional eq} to expand $\nu\left(\psi_{\alpha r + \beta}\right)$ and substitute these expressions, we obtain
\begin{flalign*}
&\quad \nu\left(\psi_{\alpha r + \beta}\right) && \\ 
&= \widetilde{\lambda}\left([\beta]P\right) - \left((ar_1 - br_2D) + c\right)^2 \widetilde{\lambda}(P) - \left((ar_2 + br_1)^2 + d\right)^2 \widetilde{\lambda}(\omega P) &&\\ 
&\quad - (ar_1 - br_2D + c)(ar_2 + br_1 + d) \left( \widetilde{\lambda}([1 + \omega]P) - \widetilde{\lambda}(P) - \widetilde{\lambda}(\omega P) \right) && \\ 
&= \widetilde{\lambda}\left([\beta]P\right) - \left( (ar_1 - br_2D)^2 + 2c(ar_1 - br_2D) + c^2 \right) \widetilde{\lambda}(P) && \\ 
&\quad - \left( (ar_2 + br_1)^2 + 2d(ar_2 + br_1) + d^2 \right) \widetilde{\lambda}(\omega P) && \\ 
&\quad - \left( (ar_1 - br_2D)(ar_2 + br_1) + d(ar_1 - br_2D) + c(ar_2 + br_1) + cd \right) && \\ 
&\quad \cdot \left( \widetilde{\lambda}([1 + \omega]P) - \widetilde{\lambda}(P) - \widetilde{\lambda}(\omega P) \right). 
\end{flalign*}

Using \Cref{functional eq}, one can see that the terms in \( c \) and \( d \) resemble \( \nu\left(\psi_{\beta}\right) \). On the other hand, by recalling that \( \alpha r = (ar_1-br_2D)+(ar_2+br_1) \) and comparing with the proof of \Cref{alphar} (Case I), we obtain \( \frac{1}{2}\nu\left(\psi_{\alpha r + 1}\psi_{\alpha r - 1}\right) \), so we get

\begin{flalign*}
&\quad \nu\left(\psi_{\alpha r + \beta}\right) && \\ 
&= \nu\left(\psi_{\beta}\right) + \frac{1}{2} \nu\left(\psi_{\alpha r + 1}\psi_{\alpha r - 1}\right) && \\ 
&\quad + 2car_1 \widetilde{\lambda}(P) - 2dar_2 \widetilde{\lambda}(\omega P) - (dar_1 + acr_2) \left( \widetilde{\lambda}([1 + \omega]P) - \widetilde{\lambda}(P) - \widetilde{\lambda}(\omega P) \right) && \\ 
&\quad + 2cbr_2D \tilde{\lambda}(P) - 2dbr_1 \widetilde{\lambda}(\omega P) + (dbr_2D - bcr_1) \left( \widetilde{\lambda}([1 + \omega]P) - \widetilde{\lambda}(P) - \widetilde{\lambda}(\omega P) \right). 
\end{flalign*}

Finally, assembling everything we obtained before and using \Cref{alphar} (Case I) for \( \nu\left(\psi_{\alpha r + 1}\psi_{\alpha r - 1}\right) \), we obtain

\begin{flalign*}
&\quad \nu\left(\psi_{\alpha r + \beta}\right) && \\ 
&= \nu\left(\psi_{\beta}\right) + a\left(\nu\left(\frac{\psi_{\beta}}{\psi_{r - \beta}}\right) + \frac{\mu(r)}{2}\right) + b\left(\nu\left(\frac{\psi_{\beta}}{\psi_{r \omega - \beta}}\right) + \frac{\mu_{\omega}(r)}{2}\right) &&\\ 
&\quad + \frac{a^{2} - ab}{2} \mu(r) + \frac{b^{2} - ab}{2} \mu_{\omega}(r) + \frac{ab}{2} \mu_{1+\omega}(r).
\end{flalign*}

\textbf{Case II:} $\omega^{2}=\omega-D$ for some square-free integer $D$.\\
For convenience, take note of the following expansions:
\begin{itemize}
\begin{multicols}{2}
\item $\alpha r=(ar_1-br_2D)+(ar_2+br_1+br_2f)\omega $
\item $r-\beta=(r_1-c)+(r_2-d)\omega$
\item $(r_1-c)(r_2-d)=r_1r_2-dr_1-cr_2+d$
\columnbreak
\item $r\omega=-r_2D+(r_1+fr_2)\omega$
\item $r\omega - \beta =(-r_2D-c)+(r_1+fr_2-d)\omega$
\item $(-r_2D-c)(r_1+fr_2-d)=c d - c f r_2 - c r_1 + d D r_2 - f r_2^2D - r_1 r_2D$
\end{multicols}
\end{itemize}

As mentioned, the proof proceeds exactly as in Case I. Note that the behaviour of $\omega$ does not affect $r-\beta$, so $a\left(\nu\left(\frac{\psi_{\beta}}{\psi_{r - \beta}}\right)\right)$ is the same as in Case I, which is 
\[
a\left(-\frac{\mu(r)}{2}\right) + 2car_1\widetilde{\lambda}(P) - 2dar_2\widetilde{\lambda}(\omega P) - (dar_1 + acr_2) \left( \widetilde{\lambda}([1 + \omega]P) - \widetilde{\lambda}(P) - \widetilde{\lambda}(\omega P) \right).
\]
Next, use \Cref{functional eq} to evaluate  and $b\left(\nu\left(\frac{\psi_{\beta}}{\psi_{r\omega - \beta}}\right)\right)$, then apply \Cref{functional eq} to $\nu\left(\psi_{\alpha r + \beta}\right)$ and assemble the terms to get the result. 

\begin{flalign*}
&\quad b\left(\nu\left(\psi_{\beta}\right)-\nu\left(\psi_{r\omega - \beta}\right)\right) &&\\
&= b\left( \widetilde{\lambda}\left([\beta]P\right) - c^2\widetilde{\lambda}(P) - d^2\widetilde{\lambda}(\omega P) - cd\left(\widetilde{\lambda}([1+\omega]P) - \tilde{\lambda}(P) - \widetilde{\lambda}(\omega P)\right) \right) &&\\
&\quad - b\left( \widetilde{\lambda}\left([r\omega-\beta]P\right) - (r_2^2D^2 + 2r_2cD + c^2)\widetilde{\lambda}\left(P\right) - (r_1^2 + 2 f r_1 r_2 + f^2 r_2^2 + d^2 - 2 d f r_2 - 2 d r_1 )\widetilde{\lambda}\left(\omega P\right) \right) &&\\
&\quad -b\left(-(c d - c f r_2 - c r_1 + d D r_2 - D f r_2^2 - D r_1 r_2)\left(\widetilde{\lambda}([1+\omega]P) - \widetilde{\lambda}(P) - \tilde{\lambda}(\omega P)\right) \right) &&\\
&= b\left(-\frac{\mu_{\omega}(r)}{2}\right)+ 2bcr_2D\widetilde{\lambda}(P) -(2bdr_1+2br_2df)\widetilde{\lambda}(\omega P) &&\\
&\quad + (bdr_2D - bcr_1-bcfr_2) \left(\widetilde{\lambda}([1 + \omega]P) - \widetilde{\lambda}(P) - \widetilde{\lambda}(\omega P) \right).
\end{flalign*}
Therefore, we have
\begin{flalign*}
&\quad \nu\left(\psi_{\alpha r + \beta}\right) && \\ 
&= \widetilde{\lambda}\left([\beta]P\right) - \left((ar_1 - br_2D) + c\right)^2 \widetilde{\lambda}(P) - \left((ar_2+br_1+br_2f)^2 + d\right)^2 \widetilde{\lambda}(\omega P) &&\\ 
&\quad - (ar_1 - br_2D + c)(ar_2+br_1+br_2f + d)\left( \widetilde{\lambda}([1 + \omega]P) - \widetilde{\lambda}(P) - \widetilde{\lambda}(\omega P) \right) && \\ 
&= \widetilde{\lambda}\left([\beta]P\right) - \left((ar_1 - br_2D)^2 + 2c(ar_1 - br_2D)+ c^2 \right) \widetilde{\lambda}(P) && \\ 
&\quad - \left((ar_2+br_1+br_2f)^2 + 2d(ar_2+br_1+br_2f)+ d^2 \right) \widetilde{\lambda}(\omega P) && \\ 
&\quad - \left((ar_1 - br_2D)(ar_2+br_1+br_2f) + d(ar_1 - br_2D) + c(ar_2+br_1+br_2f)+ cd \right) && \\ 
&\quad \cdot \left( \widetilde{\lambda}([1 + \omega]P) - \widetilde{\lambda}(P) - \widetilde{\lambda}(\omega P) \right) && \\ 
&= \nu\left(\psi_{\beta}\right) + \frac{1}{2}\nu\left(\psi_{\alpha r + 1}\psi_{\alpha r - 1}\right)&& \\ 
&\quad + 2car_1 \tilde{\lambda}(P) - 2dar_2 \tilde{\lambda}(\omega P) - (dar_1 + acr_2) \left( \widetilde{\lambda}([1 + \omega]P) - \widetilde{\lambda}(P) - \widetilde{\lambda}(\omega P) \right) && \\ 
&\quad +2bcr_2D\tilde{\lambda}(P) -(2bdr_1+2bdr_2f)\tilde{\lambda}(\omega P)  &&\\
&\quad + (bdr_2D - bcr_1-bcr_2f) \left( \widetilde{\lambda}([1 + \omega]P) - \widetilde{\lambda}(P) - \widetilde{\lambda}(\omega P) \right) &&\\
&= \nu\left(\psi_{\beta}\right) + a\left(\nu\left(\frac{\psi_{\beta}}{\psi_{r - \beta}}\right) + \frac{\mu(r)}{2}\right]+ b\left(\nu\left(\frac{\psi_{\beta}}{\psi_{r \omega - \beta}}\right)+\frac{\mu_{\omega}(r)}{2}\right)  &&\\
&\quad + \frac{a^{2} - ab}{2} \mu(r) + \frac{b^{2} - ab}{2} \mu_{\omega}(r) + \frac{ab}{2} \mu_{1+\omega}(r).
\end{flalign*}
For $\nu\left(\psi_{\alpha r - \beta}\right)$, the proof and the computation follows in the exact same way. We will do the case for $\omega^2=-D$ for some square free integer $D$ as an example:
\begin{flalign*}
&\quad \nu\left(\psi_{\alpha r - \beta}\right) && \\ 
&= \widetilde{\lambda}\left([\beta]P\right) - \left((ar_1 - br_2D) - c\right)^2 \widetilde{\lambda}(P) - \left((ar_2 + br_1)^2 - d\right)^2 \widetilde{\lambda}(\omega P) &&\\ 
&\quad - (ar_1 - br_2D - c)(ar_2 + br_1 - d)\left( \widetilde{\lambda}([1 + \omega]P) - \widetilde{\lambda}(P) - \widetilde{\lambda}(\omega P) \right) && \\ 
&= \widetilde{\lambda}\left([\beta]P\right) - \left((ar_1 - br_2D)^2 - 2c(ar_1 - br_2D)+ c^2 \right) \widetilde{\lambda}(P) && \\ 
&\quad - \left((ar_2 + br_1)^2 - 2d(ar_2 + br_1)+ d^2 \right) \widetilde{\lambda}(\omega P) && \\ 
&\quad - \left((ar_1 - br_2D)(ar_2 + br_1) - d(ar_1 - br_2D) - c(ar_2 + br_1) + cd \right) && \\ 
&\quad \cdot \left( \widetilde{\lambda}([1 + \omega]P) - \widetilde{\lambda}(P) - \widetilde{\lambda}(\omega P) \right) && \\ 
&= \nu\left(\psi_{\beta}\right)+ \frac{1}{2}\nu\left(\psi_{\alpha r + 1}\psi_{\alpha r - 1}\right) && \\ 
&\quad - 2car_1 \widetilde{\lambda}(P) + 2dar_2 \widetilde{\lambda}(\omega P) + (dar_1 + acr_2) \left( \widetilde{\lambda}([1 + \omega]P) - \widetilde{\lambda}(P) - \widetilde{\lambda}(\omega P) \right) && \\ 
&\quad - 2cbr_2D \widetilde{\lambda}(P) + 2dbr_1 \widetilde{\lambda}(\omega P) - (dbr_2D - bcr_1) \left( \widetilde{\lambda}([1 + \omega]P) - \widetilde{\lambda}(P) - \widetilde{\lambda}(\omega P) \right) &&\\
&= \nu\left(\psi_{\beta}\right) - a\left(\nu\left(\frac{\psi_{\beta}}{\psi_{r - \beta}}\right) + \frac{\mu(r)}{2}\right)- b\left(\nu\left(\frac{\psi_{\beta}}{\psi_{r \omega - \beta}}\right)+\frac{\mu_{\omega}(r)}{2}\right) &&\\
&\quad + \frac{a^{2} - ab}{2} \mu(r) + \frac{b^{2} - ab}{2} \mu_{\omega}(r) + \frac{ab}{2} \mu_{1+\omega}(r).
\end{flalign*}
\qedhere
\end{proof}

\subsection{Summary}
If $P$ has good reduction, then any multiples of $P$ will also have good reduction. \Cref{strong} tells us that the difference between the denominator ideal $D_{\bm{v}\cdot \bm{P}}^2$ and the denominator of $x(\bm{v}\cdot\bm{P})$ is exactly $F_{\bm{v}}(\bm{P})^{2}$. Using the notations in (\ref{expression}), we necessarily have 
\[
D_{z}(P)^{2}F_{\bm{z}}(\bm{P})^{-2}=\Psi_{\bm{z}}(\bm{P})^2.
\]
\textcolor{red}{Applying} the local height formula for good reduction (see (\ref{good reduction local height})) to the points $P, [\omega]P$ and $P+[\omega]P$\textcolor{red}{, we obtain}
\begin{align*}
g_{z, \nu}(P) 
&=-2\nu\left(F_{\bm{z}}(\bm{P})\right) \\
&= ( z_1 z_2-z_1^2)\max(0, -\nu(x(P))) 
   + (z_1 z_2-z_2^2)\max(0, -\nu(x([\omega]P))) \\
&\quad - (z_1 z_2)\max(0, -\nu(x([1+\omega]P)))\textcolor{red}{.}
\end{align*}

\noindent Now let $r$ be an element of the annihilator of $P$ in the $\Z[\omega]$--module $E(K)/E_{0}(K)$. If $P$ has bad reduction, but $[z]P=[\alpha r]P$ has good reduction, this is exactly the statement in \Cref{good}, which tells us that the cancellation exponent depends on if the points $[r]P, [r\omega]P$ and $[r(1+\omega)]P]$ are equivalent to the identity $\bmod \nu$. If $P$ has bad reduction, and $[z]P=[\alpha r \pm \beta]P$ also has bad reduction, then $[z]P \not\equiv O \bmod \p$. For this to be true, we must have $g_{z, \nu}(P)=2\nu(\psi_{z}(P))$, whose formula was given in \Cref{NotAnn}. Overall, when $P$ has bad reduction, the formula is 
\begin{equation}
g_{z, \nu}(P) = 
\begin{cases}
(a^{2} - ab)\mu + (b^{2} - ab)\mu_{\omega} + (ab)\mu_{1+\omega}, 
& \text{if $z = \alpha r \in \textcolor{red}{\A(\p, P)}$,} \\[0.5em]
2 \nu\!\left(\psi_{\beta}\right) 
\pm 2a\!\textcolor{red}{\left(\nu\!\left(\dfrac{\psi_{\beta}}{\psi_{r - \beta}}\right) 
   + \dfrac{\mu(r)}{2}\right)}
\pm 2b\!\textcolor{red}{\left(\nu\!\left(\dfrac{\psi_{\beta}}{\psi_{r \omega - \beta}}\right) 
   + \dfrac{\mu_{\omega}(r)}{2}\right)} \\[0.5em]
\quad + (a^{2} - ab)\mu(r) + (b^{2} - ab)\mu_{\omega}(r) + (ab)\mu_{1+\omega}(r), 
& \text{if $z = \alpha r \pm \beta \not\in \textcolor{red}{\A(\p, P)}$.}
\end{cases}
\end{equation}
where 
\[
\alpha = a + b\omega \in \Z[\omega], 
\mu(r) = g_{r, \nu}(P), \quad 
\mu_{\omega}(r) = g_{r\omega, \nu}(P), \quad 
\mu_{1+\omega}(r) = g_{r(1+\omega), \nu}(P).
\]


\section{Proof of \Cref{CM recurrence} and \Cref{general CM recurrence}}
As an application of our main result, we will now extend the work of Verzobio in \cite{Verzobio}. Let \( E/K \) be an elliptic curve with coefficients in \( \mathcal{O}_K \) and CM by an order of an imaginary quadratic field \( \mathbb{Z}[\omega] \). For a point \( P \in E(K) \), one can define a sequence of integral ideals \( \{D_{[\alpha]P}\}_{\alpha \in \mathbb{Z}[\omega]} \) that represents the square root of the denominator of \( x(\alpha P) \mathcal{O}_K \). This sequence is often referred to in the literature as an \emph{elliptic divisibility sequence}. However, even in the simplified case where \( K = \mathbb{Q} \) and \( E \) does not have CM, this sequence generally does not satisfy the recurrence relation~(\ref{recurrence}). Verzobio showed in \cite{Verzobio} that, under these conditions, by appropriately normalising the sequence \( \{D_{[n]P}\}_{n \in \mathbb{N}} \), it will satisfy (\ref{recurrence}) conditionally. 

\noindent As outlined in \cite[Section 4]{Verzobio2}, the key idea behind Verzobio's proof is that the sequence of division polynomials \( \{\psi_n\}_{n \in \mathbb{N}} \) satisfies the recurrence relation (\ref{recurrence}). This sequence is linked to \( \{D_{[n]P}\}_{n \in \mathbb{N}} \) through the cancellation exponent \( g_{n, \nu} \), as expressed in (\ref{expression}). Therefore, the most important step of the proof is to determine the conditions under which the cancellation exponent \( g_{n, \nu} \) also satisfies the recurrence relation (\ref{recurrence}). We will follow the same argument for our cancellation exponents and denominator ideals. 

\begin{prop}\label{QF recurrence}
Let $E$ be an elliptic curve defined by a Weierstrass equation with integer coefficients in $K$ and \textcolor{red}{have} complex multiplication by $\Z[\omega]$. Let $P \in E(K)$ be a non-torsion point. Let $\nu$ be a finite place and $\p$ be the associated prime. Suppose that either one of $\alpha, \beta \in \textcolor{red}{\A(\p, P)}$, or $\alpha$ and $\beta$ are both multiples of some $r\in \textcolor{red}{\A(\p, P)}$. Then
\begin{equation}
g_{\alpha + \beta, \nu}(P)+g_{\alpha-\beta, \nu}(P)=2\left(g_{\alpha, \nu}(P)+g_{\beta, \nu}(P)\right),
\end{equation}
where $g_{\alpha, \nu}(P)$ is as defined in \Cref{gzP}. 
\end{prop}

\begin{proof}
We first handle the case where \( P \) is non-singular \(\Mod \p\). In this case, we have \(d_{\p}\left(x(P\OO_{K})\right) \leq 0\). From \Cref{strong}, we know that 
\[
d_{\p}\left(D_{\bm{\alpha} \cdot \bm{P}}\right) = d_{\p}\left(\widehat{\Psi}_{\bm{\alpha}}(\bm{P})\right) = d_{\p}\left(\Psi_{\bm{\alpha}}(\bm{P})\right) + d_{\p}\left(F_{\bm{\alpha}}(\bm{P})\right).
\]
Therefore, 
\[
g_{\alpha, \nu}(P) = d_{\p}\left(D_{\bm{\alpha} \cdot \bm{P}}\right) - d_{\p}\left(\Psi_{\bm{\alpha}}(\bm{P})\right) = d_{\p}\left(F_{\bm{\alpha}}(\bm{P})\right).
\]
Recall that \( F(\bm{P})\colon \mathbb{Z}^2 \to I_{K} \) is a quadratic form, so it satisfies \( F_{\bm{\alpha} + \bm{\beta}} F_{\bm{\alpha} - \bm{\beta}} = F_{\bm{\alpha}}^2 F_{\bm{\beta}}^2 \), hence 
\[
d_{\p}\left(F_{\bm{\alpha} + \bm{\beta}}\right) + d_{\p}\left(F_{\bm{\alpha} - \bm{\beta}}\right) = 2\left(d_{\p}\left(F_{\bm{\alpha}}\right) + d_{\p}\left(F_{\bm{\beta}}\right)\right).
\]

If $P$ is singular $\Mod \p$, then according to the formula for $g_{z, \nu}(P)$, we have different cases to consider: let $\alpha_{1}=a_{1}+b_{1}\omega, \alpha_{2}=a_{2}+b_{2}\omega$ be elements in $\Z[\omega]$. Fix an element $r \in \textcolor{red}{\A(\p, P)}$.\\

\underline{Case I:} both $\alpha$ and $\beta$ are multiples of $r$ \\
Write $\alpha=\alpha_{1}r=(a_{1}+b_{1}\omega)r, \beta=\alpha_{2}r=(a_{2}+b_{2}\omega)r$. By \Cref{CM gcd}, we have 
\[
g_{\alpha, \nu}(P)=(a_1^{2}-a_1b_1)\mu(r)+(b_1^{2}-a_1b_1)\mu_{\omega}(r)+(a_1b_1)\mu_{1+\omega}(r),
\]
\text{and}
\[ 
 g_{\beta, \nu}(P)=(a_2^{2}-a_2b_2)\mu(r)+(b_2^{2}-a_2b_2)\mu_{\omega}(r)+(a_2b_2)\mu_{1+\omega}(r).
\] 
Therefore, we get 
\begin{flalign*}
&\indent g_{\alpha+\beta, \nu}(P)+g_{\alpha-\beta, \nu}(P) &&\\
&=\left((a_1+a_2)^2-(a_1+a_2)(b_1+b_2)\right)\mu(r)+\left((b_1+b_2)^2-(a_1+a_2)(b_1+b_2)\right)\mu_{\omega}(r) &&\\
&\quad +(a_1+a_2)(b_1+b_2)\mu_{1+\omega}(r) +\left((a_1-a_2)^2-(a_1-a_2)(b_1-b_2)\right)\mu(r) &&\\
&\quad +\left((b_1-b_2)^2-(a_1-a_2)(b_1-b_2)\right)\mu_{\omega}(r)+(a_1-a_2)(b_1-b_2)\mu_{1+\omega}(r) &&\\
&=(2a_1^2+2a_2^2\textcolor{red}{-}2a_1b_1\textcolor{red}{-}2a_2b_2)\mu(r)+(2b_1^2+2b_2^2\textcolor{red}{-}2a_1b_1\textcolor{red}{-}2a_2b_2)\mu_{\omega}(r) &&\\
&\quad +(2a_1\textcolor{red}{b_1}+2\textcolor{red}{a_2}b_2)\mu_{1+\omega}(r) &&\\
&=2\left(g_{\alpha, \nu}(P)+g_{\beta, \nu}(P)\right).
\end{flalign*}

\underline{Case II:} $\alpha$ is a multiple of $r$, but not $\beta$, where $\alpha=\alpha_{1}r=(a_{1}+b_{1}\omega)r, \beta=\alpha_{2}r+\gamma=(a_{2}+b_{2}\omega)r+\gamma$ for $\gamma \not\in \textcolor{red}{\A(\p, P)}$.\\
Apply \Cref{CM gcd} to get, 
\begin{flalign*}
&\quad g_{\alpha+\beta, \nu}(P)+g_{\alpha-\beta, \nu}(P)-2\left(g_{\alpha, \nu}(P)-g_{\beta, \nu}(P)\right) &&\\
&=2 \nu\left(\psi_{\gamma}\right) + 2(a_1+a_2)\left(\nu\left(\frac{\psi_{\gamma}}{\psi_{r - \gamma}}\right) + \frac{\mu(r)}{2}\right)+2(b_1+b_2)\left(\nu\left(\frac{\psi_{\gamma}}{\psi_{r \omega - \gamma}}\right)+\frac{\mu_{\omega}(r)}{2}\right) &&\\
& \quad+\left((a_1+a_2)^{2} - (a_1+a_2)(b_1+b_2)\right) \mu(r) + \left((b_1+b_2)^{2} - (a_1+a_2)(b_1+b_2)\right) \mu_{\omega}(r) &&\\
&\quad + (a_1+a_2)(b_1+b_2)\mu_{1+\omega}(r) &&\\
&\quad+2 \nu\left(\psi_{\gamma}\right) - 2(a_1-a_2)\left(\nu\left(\frac{\psi_{\gamma}}{\psi_{r - \gamma}}\right) + \frac{\mu(r)}{2}\right)-2(b_1-b_2)\left(\nu\left(\frac{\psi_{\gamma}}{\psi_{r \omega - \gamma}}\right)+\frac{\mu_{\omega}(r)}{2}\right) &&\\
& \quad+\left((a_1-a_2)^{2} - (a_1-a_2)(b_1-b_2)\right) \mu(r) + \left((b_1-b_2)^{2} - (a_1-a_2)(b_1-b_2)\right) \mu_{\omega}(r) &&\\
&\quad + (a_1-a_2)(b_1-b_2)\mu_{1+\omega}(r) &&\\
&\quad \textcolor{red}{-2}\left(2 \nu\left(\psi_{\gamma}\right) + 2a_2\left(\nu\left(\frac{\psi_{\beta}}{\psi_{r - \beta}}\right) + \frac{\mu(r)}{2}\right)+2b_2\left(\nu\left(\frac{\psi_{\beta}}{\psi_{r \omega - \beta}}\right)+\frac{\mu_{\omega}(r)}{2}\right)\right) &&\\
& \quad-2\left((a_1^{2}+a_2^2 - a_1b_1 - a_2b_2) \mu(r) + (b_1^{2}+b_2^2 - a_1b_1 - a_2b_2) \mu_{\omega}(r)+ (a_1b_1+a_2b_2)\mu_{1+\omega}(r)\right) &&\\
&=0.
\end{flalign*}

\underline{Case III:} $\alpha$ is a multiple of $r$, but not $\beta$, where $\alpha=\alpha_{1}r=(a_{1}+b_{1}\omega)r, \beta=\alpha_{2}r+\gamma=(a_{2}+b_{2}\omega)r-\gamma$ for $\gamma \not\in \textcolor{red}{\A(\p, P)}$. Repeat the proof in Case II to get the result. 

\underline{Case IV:} $\beta$ is a multiple of $r$, but not $\alpha$, where $\alpha=\alpha_{1}r+\gamma=(a_{1}+b_{1}\omega)r+\gamma, \beta=\alpha_{2}r=(a_{2}+b_{2}\omega)r$ for $\gamma \not\in \textcolor{red}{\A(\p, P)}$. Since $\Psi_{\bm{v}}$ is an odd function, this is the same as Case II.

\underline{Case V:} $\beta$ is a multiple of $r$, but not $\alpha$, where $\alpha=\alpha_{1}r-\gamma=(a_{1}+b_{1}\omega)r-\gamma, \beta=\alpha_{2}r=(a_{2}+b_{2}\omega)r$ for $\gamma \not\in \textcolor{red}{\A(\p, P)}$. This is the same as Case III.
\end{proof}

In other words, this proposition tells us that if \textcolor{red}{(\ref{QF recurrence}) holds} for all finite places, then the cancellation exponent of $\psi^2_{z}$ and $\phi_{z}$ \textcolor{red}{is a quadratic form}. 
\begin{defn}\label{M(P)}
Let $E$ be an elliptic curve defined by a Weierstrass equation with coefficients in $\OO_{K}$ and discriminant $\Delta$. We define 
\[
\mathfrak{M}(P)\coloneqq \Lcm_{\mathfrak p}\{\A(\mathfrak p, P)\}
=\textcolor{red}{\displaystyle\bigcap_{\mathfrak p \mid \Delta}} \A(\mathfrak p, P),
\]
\textcolor{red}{where $\A(\p, P)$ is as in \Cref{ANN}}. 
\end{defn}

\begin{cor}
If one of $\alpha$ and $\beta$ is an element in the ideal $\mathfrak{M}(P)$, then (\ref{QF recurrence}) holds for every finite place $\nu$. 
\end{cor}
We are almost ready to show the recurrence relation of the denominator net, after writing down one easy result about cancellation exponents.

\begin{lemma}
Let $K$ be a number field, and let $\mathfrak{a}, \mathfrak{b}, \mathfrak{c}, \mathfrak{d}$ be ideals in $\mathcal{O}_K$ such that $\mathfrak{a}$ and $\mathfrak{b}$ are coprime. 
If there exists another ideal $\mathfrak{g}$ in $\mathcal{O}_K$ satisfying 
\[
\mathfrak{g}\mathfrak{a}(\mathfrak{g}\mathfrak{b})^{-1} = \mathfrak{c}\mathfrak{d}^{-1},
\]
then
\[
\mathfrak{g} = \prod_{\mathfrak{p}\ \text{\emph{prime}}} \mathfrak{p}^{\min\{d_{\mathfrak{p}}(\mathfrak{c}),\, d_{\mathfrak{p}}(\mathfrak{d})\}}.
\]
\end{lemma}

\begin{proof}
By assumption, there is no prime ideal of $\mathcal{O}_K$ that divides both $\mathfrak{a}$ and $\mathfrak{b}$ simultaneously, so we must have 
$d_{\mathfrak{p}}(\mathfrak{a}) = 0$ or $d_{\mathfrak{p}}(\mathfrak{b}) = 0$. 
Therefore,
\[
d_{\mathfrak{p}}(\mathfrak{g})
= d_{\mathfrak{p}}(\mathfrak{g}) + \min\left(d_{\mathfrak{p}}(\mathfrak{a}),\, d_{\mathfrak{p}}(\mathfrak{b})\right)
= \min\left(d_{\mathfrak{p}}(\mathfrak{a}\mathfrak{g}),\, d_{\mathfrak{p}}(\mathfrak{b}\mathfrak{g})\right)
= \min\left(d_{\mathfrak{p}}(\mathfrak{c}),\, d_{\mathfrak{p}}(\mathfrak{d})\right),
\]
which implies the result immediately.
\end{proof}

Therefore, we have the following result immediately.

\begin{prop}\label{cancellation exponent}
Let $E$ be an elliptic curve defined by a Weierstrass equation with integer coefficients in $K$ and has complex multiplication by $\Z[\omega]$. Let $P \in E(K)$ be a non-torsion point. Let $\nu$ be a finite place and $\p$ be the associated prime. Then
\[
D_{\bm{\alpha} \cdot \bm{P}}=\Psi_{\bm{\alpha}}(\bm{P})\OO_{K}\prod_{\p\,\,\text{\emph{prime}}} \p^{-\frac{1}{2}g_{\alpha, \nu}(P)},
\]
up to multiplication by unit. 
\end{prop}

\begin{remark}
If $P$ is non-singular modulo every prime, then this proposition is identical to \Cref{strong}, which tells us that $d_{\p}\left(D_{\bm{\alpha} \cdot \bm{P}}\right)=d_{\p}\left(\hat{\Psi}_{\bm{\alpha}}(\bm{P})\right)$ at good primes.  
\end{remark}

Now we are ready to prove \Cref{general CM recurrence} and \Cref{CM recurrence} as an application of our main result (\Cref{CM gcd}). 
\begin{proof}[Proof of \Cref{general CM recurrence}]
From \Cref{cancellation exponent}, we can set 
\[
\mathfrak{g}_{\alpha} = \prod_{\mathfrak{p}\ \text{prime}} \mathfrak{p}^{\frac{1}{2} g_{\alpha, \nu}(P)}
\]
to be the sequence of \textcolor{red}{fractional} ideals $\{\g_{\alpha}\}_{\alpha \in \Z[\omega]}$ so it satisfies 
\begin{align*}
D_{\alpha}&=\psi_{\alpha}\mathfrak{g}_{\alpha}^{-1}, \,\, \text{and}\\
\psi_{\alpha+\beta}\psi_{\alpha-\beta}\psi_{\gamma}^2 &= \psi_{\alpha+\gamma}\psi_{\alpha-\gamma}\psi_{\beta}^2 - \psi_{\beta+\gamma}\psi_{\beta-\gamma}\psi_{\alpha}^2,\\
\end{align*}
with $\psi_{\alpha}$ as defined in (\ref{nett}). It remains to show that the sequence of ideals $\{\mathfrak{g}_{\alpha}\}_{\alpha \in \mathbb{Z}[\omega]}$ satisfies
\[
\mathfrak{g}_{\alpha+\beta}\mathfrak{g}_{\alpha-\beta}\mathfrak{g}_{\gamma}^2 
= \mathfrak{g}_{\alpha+\gamma}\mathfrak{g}_{\alpha-\gamma}\mathfrak{g}_{\beta}^2 
= \mathfrak{g}_{\beta+\gamma}\mathfrak{g}_{\beta-\gamma}\mathfrak{g}_{\alpha}^2.
\]
If at least two of $\alpha, \beta$ and $\gamma$ are elements of the ideal $\mathfrak{M}(P)$, then the results follow from \Cref{QF recurrence} immediately. To see this, first apply \Cref{QF recurrence} to $\mathfrak{g}_{\alpha+\beta}\mathfrak{g}_{\alpha-\beta}$, then we obtain
\begin{align*}
\mathfrak{g}_{\alpha+\beta}\mathfrak{g}_{\alpha-\beta}\mathfrak{g}_{\gamma}^2  
&= \prod_{\mathfrak{p}\ \text{prime}} 
   \mathfrak{p}^{\frac{1}{2}\left(g_{\alpha+\beta, \nu}(P) + g_{\alpha-\beta, \nu}(P)\right) + g_{\gamma, \nu}(P)}  \\ 
&= \prod_{\mathfrak{p}\ \text{prime}} 
   \mathfrak{p}^{\frac{1}{2}\left(2g_{\alpha, \nu}(P) + 2g_{\beta, \nu}(P)\right) + g_{\gamma, \nu}(P)}  \\ 
&= \prod_{\mathfrak{p}\ \text{prime}} 
   \mathfrak{p}^{g_{\alpha, \nu}(P) + g_{\beta, \nu}(P) + g_{\gamma, \nu}(P)}. 
\end{align*}
For $\mathfrak{g}_{\alpha+\gamma}\mathfrak{g}_{\alpha-\gamma}\mathfrak{g}_{\beta}^2 $, we apply \Cref{QF recurrence} to $\mathfrak{g}_{\alpha+\gamma}\mathfrak{g}_{\alpha-\gamma}$ instead and similarly for $\mathfrak{g}_{\beta+\gamma}\mathfrak{g}_{\beta-\gamma}\mathfrak{g}_{\alpha}^2$.
\qedhere
\end{proof}

\begin{proof}[Proof of \Cref{CM recurrence}]
By definition of elliptic net, the sequence $\{\psi_{\alpha}\}_{\alpha \in \Z[\omega]}$ satisfies
\[
\psi_{\alpha+\beta}\psi_{\alpha-\beta}\psi_{\gamma}^2=\psi_{\alpha+\gamma}\psi_{\alpha-\gamma}\psi_{\beta}^2-\psi_{\beta+\gamma}\psi_{\beta-\gamma}\psi_{\alpha}^2\,\, \text{for any $\alpha, \beta, \gamma \in \Z[\omega]$.}
\]
Following the method in \cite[Theorem 1.9]{Verzobio}, we divide both sides by 
\[
\prod_{p\,\,\text{\emph{prime}}} p^{-\frac{1}{2} \left( g_{\alpha, \nu} + g_{\beta, \nu} + g_{\gamma, \nu} \right)}.
\]
But this is equivalent to multiplying both sides by $\mathfrak{g}_{\alpha+\beta}\mathfrak{g}_{\alpha-\beta}\mathfrak{g}_{\gamma}^2 $ since \textcolor{red}{$\OO_K$} is a principal ideal domain, which means all prime ideals are principal ideals generated by prime elements (i.e. $\p=(p)$).  Furthermore, in \Cref{general CM recurrence} we have also shown that 
\[
\mathfrak{g}_{\alpha+\beta}\mathfrak{g}_{\alpha-\beta}\mathfrak{g}_{\gamma}^2 = \mathfrak{g}_{\alpha+\gamma}\mathfrak{g}_{\alpha-\gamma}\mathfrak{g}_{\beta}^2 = \mathfrak{g}_{\beta+\gamma}\mathfrak{g}_{\beta-\gamma}\mathfrak{g}_{\alpha}^2.
\]
Therefore, we obtain
\[
\frac{\psi_{\alpha+\beta}\psi_{\alpha-\beta}\psi_{\gamma}^2}{\mathfrak{g}_{\alpha+\beta}\mathfrak{g}_{\alpha-\beta}\mathfrak{g}_{\gamma}^2}
=\frac{\psi_{\alpha+\gamma}\psi_{\alpha-\gamma}\psi_{\beta}^2}{\mathfrak{g}_{\alpha+\gamma}\mathfrak{g}_{\alpha-\gamma}\mathfrak{g}_{\beta}^2} 
- \frac{\psi_{\beta+\gamma}\psi_{\beta-\gamma}\psi_{\alpha}^2}{\mathfrak{g}_{\beta+\gamma}\mathfrak{g}_{\beta-\gamma}\mathfrak{g}_{\alpha}^2}.
\]
Recall that we have chosen a sequence of generators $D_{[\alpha]P}=(B_{\alpha})$, so by \Cref{general CM recurrence}, we have $B_{\alpha}=\psi_{\alpha}\g_{\alpha}^{-1}$ as elements in $K$ and 
\[
B_{\alpha+\beta}B_{\alpha-\beta}B_{\gamma}^2=B_{\alpha+\gamma}B_{\alpha-\gamma}B_{\beta}^2-B_{\beta+\gamma}B_{\beta-\gamma}B_{\alpha}^2
\]
as required. 
\end{proof}
\subsubsection*{Acknowledgements}
This work was undertaken under the supervision of Dr.\ Simon L.\ Rydin Myerson, and Dr.\ Helena Verill, whom the author thanks for their continued guidance and support, as well as for many discussions and insightful comments on this work. The author would also like to thank the anonymous referee for the careful reading of the manuscript and many helpful comments
\appendix
\section{Examples for illustrating \Cref{CM gcd} and \Cref{general CM recurrence}}\label{Examples}

\begin{example}\label{Example 1}
We consider the elliptic curve 
\(
E \colon y^{2} = x^{3} - 2x
\)
over the field \( \mathbb{Q}(i) \), which has complex multiplication by 
\(\mathbb{Z}[i] = \mathbb{Z}[\omega]\) and discriminant 
\(\Delta = (1+i)^{18}\).
The group \(E\big(\mathbb{Q}(i)\big)\) is generated by the two points 
\(P = (-1, 1)\) and \(\omega P = (1, i)\).
The curve has only one bad prime, namely \(1+i\), but both \(P\) and \([i]P\) have good reduction at this prime. 
Thus, throughout this example we are always in the setting of \Cref{CM gcd}(1).
Equivalently, the common valuations between 
\(\Psi^{2}_{\bm{v}}(\bm{P})\) and \(\Phi_{\bm{v}}(\bm{P})\) depend solely on the valuation of the quadratic form \(F_{\bm{v}}(\bm{P})\). Therefore, we omit the computation of \(\Phi_{\bm{v}}(\bm{P})\).

Furthermore, since \(\mathbb{Z}[i]\) is a principal ideal domain, \(\widehat{\Psi}_{\bm{v}}(\bm{P})\) is a scaled elliptic net associated to \(\Psi_{\bm{v}}(\bm{P})\), and \Cref{Table 1} shows that 
\(
\widehat{\Psi}_{\bm{\alpha}}(\bm{P}) = B_{\alpha}
\)
for all \(\alpha \in \mathbb{Z}[i]\), where \(B_{\alpha}\) is the generator of the denominator ideal \(D_{[\alpha]P}\). This reflects the fact that the point \(P\) has good reduction at all primes; hence the constant \(M(P)\) defined in \Cref{M(P)} is simply \(1\). Since \(\widehat{\Psi}_{\bm{\alpha}}(\bm{P})\) is an elliptic net, it satisfies the recurrence relation in \Cref{elliptic net} for all \(\alpha \in \mathbb{Z}[i]\), and therefore the sequence of generators \(\{B_{\alpha}\}_{\alpha \in \mathbb{Z}[i]}\) satisfies the recurrence relation stated in \Cref{CM recurrence}. More specifically, in this case the sequence of ideals \(\{\mathfrak{g}_{\alpha}\}_{\alpha \in \mathbb{Z}[i]}\) in \Cref{general CM recurrence} is simply the sequence of quadratic forms \(\{F_{\bm{\alpha}}(\bm{P})\}_{\alpha \in \mathbb{Z}[i]}\).
\end{example}

\begin{example}\label{Example 2}
We consider the elliptic curve 
\(
E \colon y^2 = x^3 + x^2 - 3x + 1
\)
over the field \( \mathbb{Q}(\sqrt{-2}) \), which has complex multiplication by 
\(\mathbb{Z}[\sqrt{-2}] = \mathbb{Z}[\omega]\) and discriminant 
\(\Delta = 512 = 2^{10} = (\sqrt{-2})^{20}\).
The group \(E\big(\mathbb{Q}(\sqrt{-2})\big)\) is generated by the two points 
\(P = (-1, 2)\) and 
\(\omega P = \left( -\frac{1}{\omega^2}, \frac{1}{\omega^3} \right)\).
The points \(P\) and \(P + [\omega]P\) have singular reduction modulo \(\sqrt{-2}\), whereas \([\omega]P\) does not. Since \(\mathbb{Z}[\sqrt{-2}]\) is a principal ideal domain, we have 
\(\Ann(P) = (\sqrt{-2})\), and we choose \(r = \sqrt{-2}\). 
Let \(\nu\) be the valuation corresponding to \(r = \sqrt{-2}\). Then
\[
\mu(r) = \nu(\phi_{\omega}) = -2, \qquad 
\mu_{\omega}(r) = \nu(\phi_{-2}) = \nu(20) = 4,
\]
and
\[
\mu_{1+\omega}(r) 
  = \nu(\phi_{-2+\omega})
  = \nu\!\left( \frac{(1-\omega)^4(-3 + 8\omega)}{2} \right)
  = -2.
\]

In \Cref{Table 2} we list the coordinates of several points in \(E\big(\mathbb{Q}(\sqrt{-2})\big)\), 
the values of the net polynomials \(\Psi_{\bm{v}}(\bm{P})\) and \(\Phi_{\bm{v}}(\bm{P})\) 
associated to \(E\) and the points \(P\) and \(\omega P\), as well as the generators \(B_{\alpha}\) 
of the denominator ideals. Note that the only bad prime for \(E/\mathbb{Q}(\sqrt{-2})\) is \(\sqrt{-2}\).
As predicted by \Cref{CM gcd}(1), for points \([\alpha]P \in E\big(\mathbb{Q}(\sqrt{-2})\big)\) 
with \(\alpha\) a multiple of \(\sqrt{-2}\), the valuation of the quadratic form 
\(F_{\bm{v}}(\bm{P})\) coincides with \(g_{\alpha,\nu}(P)\) at all primes except the one corresponding to \(\sqrt{-2}\). In this case, \Cref{CM gcd}(2) correctly determines the cancellation exponent. 

For example, consider \([2+2\sqrt{-2}]P\).  
At the prime \(1-\sqrt{-2}\) we have \(g_{2+\sqrt{-2},\,1-\sqrt{-2}}(P) = -8\), which equals the valuation of \(F_{(2,2)}^{-2}(\bm{P})\).  
On the other hand, at the bad prime \(\sqrt{-2}\) we obtain \(g_{2+\sqrt{-2},\,\sqrt{-2}}(P) = 4\), a contribution entirely absent from \(F_{(2,2)}(\bm{P})^{2}\).

As in \Cref{Example 1}, in order for the sequence of generators \(\{B_{\alpha}\}_{\alpha \in \mathbb{Z}[\sqrt{-2}]}\) of the denominator ideals \(\{D_{[\alpha]P}\}_{\alpha \in \mathbb{Z}[\sqrt{-2}]}\) to satisfy the recurrence relation in \Cref{CM recurrence}, we require 
\(\widehat{\Psi}_{\bm{\alpha}}(\bm{P}) = B_{\alpha}\).
This occurs precisely when \([\alpha]P\) has good reduction everywhere, so we must have 
\(\alpha \in (\sqrt{-2}) = \mathfrak{M}(P)\).
Once again, the sequence of ideals \(\{\mathfrak{g}_{\alpha}\}_{\alpha \in \mathbb{Z}[\sqrt{-2}]}\) 
in \Cref{general CM recurrence} is simply the sequence of quadratic forms 
\(\{F_{\bm{\alpha}}(\bm{P})\}_{\alpha \in \mathbb{Z}[\sqrt{-2}]}\).
\end{example}

\newgeometry{top=0.4in,bottom=0.4in,left=0.75in,right=0.75in}
\begin{landscape}
\pagestyle{empty}

\begin{table}[H]
    \centering
    \arrayrulecolor{black}
    \setlength{\arrayrulewidth}{1pt}
    \renewcommand\arraystretch{1.5}
    \makebox[\linewidth]{%
    \begin{tabular}{c|c|c|c|c|c}
        $\alpha \in \Z[i]$ & $[\alpha] P$ & $B_{\alpha}$ & $\Psi_{\bm{v}}(\bm{P})$ & $F_{\bm{v}}(\bm{P})$ & $\hat{\Psi}_{\bm{v}}(\bm{P})$ \\
        \hline
        $1$ & $(-1, 1)$ & $1$ & $1$ & $1$ & $1$ \\
        $i$ & $(1, i)$ & $1$ & $1$ & $1$ & $1$ \\
        $1+i$ & $\left(-\frac{i}{2}, \frac{-3i-3}{4}\right)$ & $1+i$ & $1$ & $1+i$ & $1+i$ \\
        $1-i$ & $\left(\frac{i}{2}, \frac{3i-3}{4}\right)$ & $1+i$ & $2$ & $(1+i)^{-1}$ & $1-i$ \\
        $1+2i$ & $\left(\frac{(1+4i)^{2}}{(2+i)^{2}}, -\frac{(4+i)(16+9i)}{(1+2i)^{3}}\right)$ & $2+i$ & $1+\frac{i}{2}$ & $(1+i)^{2}$ & $i(2-i)$ \\
        $1-2i$ & $\left(\frac{(4+i)^{2}}{(1+2i)^{2}}, -\frac{(4+i)(16+9i)}{(1+2i)^{3}}\right)$ & $2-i$ & $2(2-i)$ & $(1+i)^{-2}$ & $-(1+2i)$ \\        
        $2$ & $\left(\frac{9}{4}, -\frac{21}{8}\right)$ & $2$ & $2$ & $1$ & $2$ \\
        $2i$ & $\left(-\frac{9}{4}, -\frac{21}{8}i\right)$ & $2$ & $2i$ & $1$ & $2i$ \\    
        $2+i$ & $\left(-\frac{(4+i)^{2}}{(1+2i)^{2}}, \frac{(4+i)(16+9i)}{i(1+2i)^{3}}\right)$ & $2-i$ & $-1+\frac{i}{2}$ & $(1+i)^{2}$ & $-i(2-i)$ \\
        $2-i$ & $\left(-\frac{(1+4i)^{2}}{(2+i)^{2}}, \frac{(4+i)(16+9i)}{i(1+2i)^{3}}\right)$ & $2+i$ & $2(2+i)$ & $(1+i)^{-2}$ & $-i(2+i)$ \\
        $2+2i$ & $\left(-\frac{7^{2}}{3^{2}(1+i)^{6}}, \frac{(8-7i)(8+7i)}{3^{3}(1+i)^{9}}\right)$ & $3(1+i)^{3}$ & $-\frac{3}{1-i}$ & $(1+i)^{4}$ & $-3i(1+i)^{3}$ \\
        $2-2i$ & $\left(\frac{7^{2}}{3^{2}(1+i)^{6}}, \frac{-i(8-7i)(8+7i)}{3^{3}(1+i)^{9}}\right)$ & $3(1+i)^{3}$ & $-3(1+i)^{7}$ & $(1+i)^{-4}$ & $-3(1+i)^{3}$ \\
        $3$ & $\left(-\frac{1}{169}, \frac{239}{2197}\right)$ & $(3+2i)(3-2i)$ & $-13$ & $1$ & $-13$ \\
        $3+i$ & $\left(\frac{(9+16i)^{2}}{(1+i)^{2}(2+i)^{2}(1+4i)^{2}}, \frac{3(3+8i)(5+8i)(9+16i)}{(1+i)^{3}(2+i)^{3}(1+4i)^{3}}\right)$ & $(1+i)(2+i)(4-i)$ & $-\frac{(2+i)(1+4i)}{(1+i)^{2}}$ & $(1+i)^{3}$ & $-i(1+i)(4-i)(2+i)$\\
        $3-i$ & $\left(\frac{(16+9i)^{2}}{(1+i)^{2}(1+2i)^{2}(4+i)^{2}}, \frac{3(3+8i)(5+8i)(9+16i)}{(1+i)^{3}(2+i)^{3}(1+4i)^{3}}\right)$ & $(1+i)(2+i)(4-i)$ & $-i(1+2i)(4+i)(1+i)^{4}$ & $(1+i)^{-3}$ & $-i(1+i)(4+i)(1+2i)$\\
        $3+2i$ & $\left(-\frac{(31+20i)^{2}}{(1+6i)^{2}(5+4i)^{2}}, -\frac{(31+20i)(999+1360i)}{(1+6i)^{3}(5+4i)^{3}}\right)$ & $(5+4i)(6-i)$ & $\frac{(6-i)(5+4i)}{i(1+i)^{6}}$ & $(1+i)^{6}$ & $i(6-i)(5+4i)$ \\
        $3+3i$ & $\left(\frac{239^{2}}{(1+i)^{2}(3+2i)^{2}(3-2i)^{2}}, \frac{i(3^{2})(11)(239)(24+i)(1+24i)}{(1+i)^{3}(3+2i)^{3}(2+3i)^{3}}\right)$ & $(1+i)(3+2i)(3-2i)$ & $-\frac{13}{16}$ & $(1+i)^{9}$ & $-i(1+i)(3+2i)(2+3i)$
    \end{tabular}}
\caption{Comparison of the values of the denominator ideal generators \( \left( B_{\alpha} \right)_{\alpha \in \mathbb{Z}[\omega]} \) and the elliptic nets \( \Psi_{\bm{v}}(\bm{P}), \Phi_{\bm{v}}(\bm{P}) \) associated with the curve $E \colon y^{2}=x^{3}-2x$ over $\Q(i)$, where \( \bm{P} = (P_1, P_2) = (P, \omega P) \), and \( P = (-1, 1) \). For \( \alpha = a + b\omega \), we take \( \bm{v} = (a, b) \).
}
\label{Table 1}
\end{table}

\newpage

\pagestyle{empty}
\begin{table}[H]
    \vfill
    \centering
    \arrayrulecolor{black}
    \setlength{\arrayrulewidth}{1pt}
    \renewcommand\arraystretch{1.5}
    \adjustbox{max width=\linewidth}{%
    \begin{tabular}{c|c|c|c|c|c|c}
        $\alpha \in \mathbb{Z}[\omega]$ & $[\alpha] P$ & $B_{\alpha}$ & $\Psi_{\bm{v}}(\bm{P})$ & $F_{\bm{v}}(\bm{P})$ & $\widehat{\Psi}_{\bm{v}}(\bm{P})$ & $\Phi_{\bm{v}}(\bm{P})$\\
        \hline
        $1$ & $(-1, 2)$ & $1$ & $1$ & $1$ & $1$ & $-1$\\
        $\omega$ & $\left(-\frac{1}{\omega^2}, \frac{1}{\omega^3}\right)$ & $\omega$ & $1$ & $\omega$ & $\omega$    &   $-\frac{1}{\omega^2}$ \\
        $1+\omega$ & $\left(\frac{(-3-2\omega)}{(1-\omega)^{2}}, \frac{-22+26\omega}{3^3}\right)$ & $1-\omega$ & $1$ & $1-\omega$ & $1-\omega$  &   $-\frac{3+2\omega}{(1-\omega)^2}$ \\
        $1-\omega$ & $\left(\frac{-3+2\omega}{(1+\omega)^{2}}, \frac{-22-26\omega}{3^3}\right)$ & $1+\omega$ & $-\frac{(1+\omega)(1-\omega)}{\omega^2}$ & $\frac{\omega^2}{1-\omega}$ & $-(1+\omega)$  &   $\frac{(1-\omega)^2(3-2\omega)}{\omega^4}$ \\
        $1+2\omega$ &   $\left(\frac{7(9-4\omega)}{(1+\omega)^2(3-\omega)^2}, \frac{-84196\omega-99842}{3^3\cdot11^3}\right)$    &   $(1+\omega)(3-\omega)$  &  $-\frac{(-3+\omega)(1+\omega)}{\omega^2(1-\omega)^2}$ &   $\omega^2(1-\omega)^2$   &   $-(3-\omega)(1+\omega)$  &   $\frac{7(9-4\omega)}{\omega^4(1-\omega^4)}$\\
        $1-2\omega$ &   $\left(\frac{7(9+4\omega)}{(1-\omega)^2(3+\omega)^2}, \frac{84196\omega - 99842}{3^3\cdot 11^3}\right)$  &   $(1-\omega)(3+\omega)$  &   $-\frac{(1-\omega)^3(3+\omega)}{\omega^6}$  &  $\frac{\omega^6}{(1-\omega)^2}$ &   $-(1-\omega)(3+\omega)$	&	$\frac{7(1-\omega)^4(9-4\omega)}{\omega^{12}}$ \\
        $2$ &   $\left(\frac{5}{4}, \frac{7}{8}\right)$ &   $\omega^2$    & $\omega^4$ &   $1$ &  $\omega^4$ &    $20$\\
        $2\omega$   &   $\left(-\frac{41}{8}, -\frac{217}{32\omega}\right)$ &    $\omega^3$  &   $-\frac{1}{\omega}$ &   $\omega^4$  &   $-\omega^3$	&	$\frac{41}{16}$\\
        $2+\omega$  &   $\left(\frac{3+8\omega}{\omega^2(1+\omega)^2}, \frac{5\omega+322}{3^3\cdot 2^2}\right)$ &   $\omega(1+\omega)$  &   $\frac{\omega^4(1+\omega)}{(1-\omega)^2}$ &   $\frac{(1-\omega)^2}{\omega}$   & $\omega^3(1+\omega)$    &   $\frac{\omega^6(3+2\omega)^2}{(1-\omega)^4}$  \\
        $2-\omega$  &   $\left(\frac{3-8\omega}{\omega^2(1-\omega)^2}, \frac{332-5\omega}{108}\right)$    &   $\omega(1-\omega)$  &  $-(1-\omega)^3$   & $\frac{\omega^3}{(1-\omega)^2}$ &   $\omega^3(1-\omega)$   &   $\frac{(1-\omega)^4(-3+8\omega)}{2}$\\
        $2+2\omega$ &  $\left(\frac{-147-32\omega}{\omega^4(1-\omega)^2(3-2\omega)^2}, \frac{656425\omega - 179123}{2^3\cdot 3^3 \cdot 17^3}\right)$ &  $\omega^2(1-\omega)(3-2\omega)$ &   $-\frac{\omega^4(3-2\omega)}{(1-\omega)^3}$   &  $(1-\omega)^4$  &   $\omega^4(1-\omega)(3-2\omega)$  &	$\frac{\omega^4(-147-32\omega)}{(1-\omega)^8}$\\
    \end{tabular}}
\caption{Comparison of the values of the denominator ideal generators \( \left( B_{\alpha} \right)_{\alpha \in \mathbb{Z}[\omega]} \) and the elliptic nets \( \Psi_{\bm{v}}(\bm{P}), \Phi_{\bm{v}}(\bm{P}) \) associated with the curve \( E: y^2 = x^3 + x^2 - 3x + 1 \) over \( \mathbb{Q}(\sqrt{-2}) \), where \( \bm{P} = (P_1, P_2) = (P, \omega P) \), and \( P = (-1, 2) \). For \( \alpha = a + b\omega \), we take \( \bm{v} = (a, b) \).
}
\label{Table 2}
\vfill
\end{table}
\begin{remark}
Note that this curve is isomorphic to the curve 
\(E \colon y^2 = x^3 + 4x^2 + 2x\) given by Silverman in 
\cite[Proposition~2.3.1, p.~111]{Silverman2}, via the change of variables \(x = x' + 1\).
\end{remark}
\end{landscape}
 \restoregeometry

\bibliographystyle{abbrv}
\bibliography{Reference2}

@article{Akbary,
title = {On symmetries of elliptic nets and valuations of net polynomials},
journal = {Journal of Number Theory},
volume = {158},
pages = {185-216},
year = {2016},
issn = {0022-314X},
doi = {https://doi.org/10.1016/j.jnt.2015.06.005},
url = {https://www.sciencedirect.com/science/article/pii/S0022314X15002188},
author = {Amir Akbary and Jeff Bleaney and Soroosh Yazdani}
}

@article{Amir,
author = {Ghadermarzi, Amir},
year = {2023},
month = {01},
pages = {121-159},
title = {Multiples of integral points on Mordell curves},
volume = {211},
journal = {Acta Arithmetica},
doi = {10.4064/aa220822-3-8}
}

@article{Ayad,
author = {Ayad, Mohamed},
journal = {Manuscripta Mathematica},
number = {3-4},
pages = {305-324},
title = {Points s-entiers des courbes elliptiques.},
url = {http://eudml.org/doc/155756},
volume = {76},
year = {1992},
}

@misc{Panda, 
	author = {Corina Bianca Panda}, 
	title = {N\'{e}ron local height functions for elliptic curves},
	publisher={Universit\'{e} Paris-Sud 11, Universitet Leiden},
	howpublished = {Available at 
	\url{https://www.algant.eu/documents/theses/panda.pdf} 
	(2013)} 
}

@article{Cheon,
title = {Explicit valuations of division polynomials of an elliptic curve},
journal = {Manuscripta Mathematica},
volume = {158},
pages = {97},
year = {1998},
doi = {https://doi.org/10.1007/s002290050104},
author = {J. Cheon and S. Hahn}
}

@book{Cox,
  author    = {David A. Cox},
  title     = {Primes of the Form $x^{2} + ny^{2}$: Fermat, Class Field Theory, and Complex Multiplication},
  edition   = {2nd},
  year      = {2013},
  publisher = {John Wiley \& Sons, Inc.},
  address   = {Hoboken, New Jersey}
}

@book{Serre, 
	author = {Jean-Pierre Serre}, 
	title = {Lectures on the Mordell-Weil Theorem}, 
	edition = {3}, 
	publisher = {Vieweg+Teubner Verlag Wiesbaden}, 
	year = {2013} ,
	doi={https://doi.org/10.1007/978-3-663-10632-6}
}

@article{Stange,
   title={Elliptic nets and elliptic curves},
   volume={5},
   ISSN={1937-0652},
   url={http://dx.doi.org/10.2140/ant.2011.5.197},
   DOI={10.2140/ant.2011.5.197},
   number={2},
   journal={Algebra \& amp; Number Theory},
   publisher={Mathematical Sciences Publishers},
   author={Stange, Katherine},
   year={2011},
   month=aug, pages={197–229} 
}

@article{Stange2, 
	title={Integral Points on Elliptic Curves and Explicit Valuations of Division Polynomials}, 
	volume={68}, 
	DOI={10.4153/CJM-2015-005-0}, 
	number={5}, 
	journal={Canadian Journal of Mathematics}, 
	author={Stange, Katherine E.}, 
	year={2016}, 
	pages={1120–1158}
}

@misc{Stange3,
  author       = {Stange, Katherine E.},
  title        = {Formulary for elliptic divisibility sequences and elliptic nets},
  howpublished = {\url{https://math.colorado.edu/~kstange/papers/edsformulary.pdf}}
}

@article{Stange4, 
	title={Division polynomials for arbitrary isogenies}, 
	volume={12}, 
	DOI={https://doi.org/10.1007/s40993-026-00741-2}, 
	number={53}, 
	journal={Research in Number Theory}, 
	author={Katherine E. Stange}, 
	year={2026}
}

@article{Streng, 
	title={Divisibility sequences for elliptic curves with complex multiplication}, 
	volume={2}, 
	DOI={http://dx.doi.org/10.2140/ant.2008.2.183}, 
	number={2}, 
	journal={Algebra \& Number Theory}, 
	author={Marco Streng}, 
	year={2008}, 
	pages={183-208}
}

@article{Verzobio,
  title={A recurrence relation for elliptic divisibility sequences},
  volume={13},
  number={1},
  journal={Rivista di Matematica della Università di Parma},
  author={Matteo Verzobio},
  year={2022},
  pages={223--242},
  url={www.rivmat.unipr.it/vols/2022-13-1/12-verzobio.html}
}

@article{Verzobio2,
	 title={Common valuations of division polynomials}, 
	DOI={10.1017/prm.2024.7}, journal={Proceedings of the Royal Society of Edinburgh: Section A Mathematics}, 
	author={Naskr\k{e}cki, Bartosz and Verzobio, Matteo}, 
	year={2024}, 
	pages={1–15}
}

@article{Ingram,
title = {Multiples of integral points on elliptic curves},
journal = {Journal of Number Theory},
volume = {129},
number = {1},
pages = {182-208},
year = {2009},
issn = {0022-314X},
doi = {https://doi.org/10.1016/j.jnt.2008.08.001},
url = {https://www.sciencedirect.com/science/article/pii/S0022314X08001492},
author = {Patrick Ingram},
}

@book{Silverman2, 
	author = {Joseph H. Silverman}, 
	title = {Advanced topics in the arithmetic of elliptic curves}, 
	url = {https://doi.org/10.1007/978-1-4612-0851-8},
	edition = {1}, 
	publisher = {Springer New York, NY}, 
	year = {1994}
}

@book{Silverman1, 
	author = {Joseph H. Silverman}, 
	title = {The arithmetic of elliptic curves}, 
	url = {https://doi.org/10.1007/978-0-387-09494-6},
	edition = {2}, 
	publisher = {Springer New York, NY}, 
	year = {2009}
}

@article{Brauer,
	author = {Bhakta, Subham and Loughran, Daniel and Rydin Myerson, Simon L. and Nakahara, Masahiro},
	title = {The elliptic sieve and Brauer groups},
	journal = {Proceedings of the London Mathematical Society},	
	volume = {126},
	number = {6},
	pages = {1884-1922},
	doi = {https://doi.org/10.1112/plms.12520},
	url = {https://londmathsoc.onlinelibrary.wiley.com/doi/abs/10.1112/plms.12520},
	eprint = {https://londmathsoc.onlinelibrary.wiley.com/doi/pdf/10.1112/plms.12520},
year = {2023}
}

@article{Satoh, 
	title={Generalized division polynomials}, 
	volume={94}, url={https://www.mscand.dk/article/view/14436}, 
	DOI={10.7146/math.scand.a-14436}, 
  abstractNote={Let $E$ be an elliptic curve with complex multiplication by the ring $O_{F}$ of integers of an imaginary quadratic field $F$. We give an explicit condition on $\alpha\in O_{F}$ so that there exists a rational function $\psi_{\alpha}$ satisfying $\operatorname{div}(\psi_{\alpha})=\sum_{P\in\mathrm{Ker}[\alpha]}[P] - N_{F / \mathbb{Q}}(\alpha)[\mathcal{O}]$, where $[\alpha]$ is the multiplication by $\alpha$ map. We give an algorithm to compute $\psi_{\alpha}$ based on recurrence formulas among these functions. We prove that the time complexity of this algorithm is $O(N_{F/\mathbb{Q}}(\alpha)^{2+\varepsilon})$ bit operations under an FFT-based multiplication algorithm as $N_{F/\mathbb{Q}}(\alpha)$ tends to infinity for the fixed $E$.},
	number={2}, 
	journal={MATHEMATICA SCANDINAVICA}, 
	author={Satoh, Takakazu}, 
	year={2004},
	month={Jun.},
	pages={161–184} 
}

@article{Yabuta,
   title={The greatest common valuation of $\phi_n$ and $\psi_{n}^2$ at points on elliptic curves},
   volume={229},
   ISSN={0022-314X},
   url={http://dx.doi.org/10.1016/j.jnt.2021.04.015},
   DOI={10.1016/j.jnt.2021.04.015},
   journal={Journal of Number Theory},
   publisher={Elsevier BV},
   author={Voutier, Paul and Yabuta, Minoru},
   year={2021},
   month=dec, pages={16–38} }

\end{document}